\documentclass[onefignum,onetabnum]{siamart190516}

\usepackage{lipsum}
\usepackage{amsfonts}
\usepackage{graphicx}
\usepackage{epstopdf}
\usepackage{algorithmic}
\ifpdf
  \DeclareGraphicsExtensions{.eps,.pdf,.png,.jpg}
\else
  \DeclareGraphicsExtensions{.eps}
\fi

\newsiamremark{remark}{Remark}
\newsiamremark{hypothesis}{Hypothesis}
\crefname{hypothesis}{Hypothesis}{Hypotheses}
\newsiamthm{claim}{Claim}

\headers{Learning Hidden Dynamics}{K.~Xu, D.~Li, E.~Darve and J.~M.~Harris}

\newcommand\titlelowercase[1]{\texorpdfstring{\lowercase{#1}}{#1}}
\title{Learning Hidden Dynamics using Intelligent Automatic Differentiation\footnote{K\titlelowercase{ailai} X\titlelowercase{u and} D\titlelowercase{ongzhuo} L\titlelowercase{i contributed equally to this paper.}}}

\author{Kailai~Xu\thanks{Institute for Computational and Mathematical Engineering, Stanford University, Stanford, CA, 94305 (\email{kailaix@stanford.edu})}
\and Dongzhuo~Li\thanks{Department of Geophysics, Stanford University, Stanford, CA, 94305 (\email{lidongzh@stanford.edu})}
\and Eric~Darve\thanks{Mechanical Engineering and Institute for Computational and Mathematical Engineering, Stanford University, Stanford, CA, 94305 (\email{darve@stanford.edu})}
\and Jerry~M.~Harris\thanks{Department of Geophysics and Institute for Computational and Mathematical Engineering, Stanford University, Stanford, CA, 94305 (\email{harrisgp@stanford.edu})}
}

\usepackage{amsopn}

\usepackage{titlecaps}
\Addlcwords{the of into via for and of on in an to hp-finite with}

\usepackage{amsmath}
\usepackage{amssymb}
\usepackage{tikz}
\usepackage{mathtools}
\usepackage{bm}
\usetikzlibrary{patterns}
\usepackage{subcaption}
\usepackage{pgfplots}
\pgfplotsset{compat=newest}
\usepackage{float}
\usepackage{indentfirst}
\usepackage[]{algorithm}

\theoremstyle{plain}

\usepackage{makecell}
\usepackage{booktabs}
\usepackage{todonotes}

\newcommand{\revise}[1]{{#1}}

\newcommand{\RR}[0]{\mathbb{R}}

\newcommand{\bx}{\mathbf{x}}
\newcommand{\by}{\mathbf{y}}

\newcommand{\bc}{\mathbf{c}}

\renewcommand{\bf}{\mathbf{f}}
\newcommand{\bu}{\mathbf{u}}

\newcommand{\bv}[0]{\mathbf{v}}
\newcommand{\bt}[0]{\bm{\theta}}

\begin{document}

\maketitle


\begin{abstract}
Many engineering problems involve learning hidden dynamics from indirect observations, where the physical processes are described by systems of partial differential equations (PDE). Gradient-based optimization methods are considered scalable and efficient to learn hidden dynamics. However, one of the most time-consuming and error-prone tasks is to derive and implement the gradients, especially in systems of PDEs where gradients from different systems must be correctly integrated together. To that purpose, we present a novel technique, called intelligent automatic differentiation (IAD), to leverage the modern machine learning tool \texttt{TensorFlow} for computing gradients automatically and conducting optimization efficiently. Moreover, IAD allows us to integrate specially designed state adjoint method codes to achieve better performance. Numerical tests demonstrate the feasibility of IAD for learning hidden dynamics in complicated systems of PDEs; additionally, by incorporating custom built state adjoint method codes in IAD, we significantly accelerate the forward and inverse simulation.
\end{abstract}

\section{Introduction}
Learning hidden dynamics from indirect observations --- observations that depend on but are not the states of the dynamics --- is fundamental in many applications. For example, in medical imaging, the reaction-advection-diffusion equation is used for modeling the hidden dynamics of tumor concentration. The coefficients in the reaction-advection-diffusion equation are not known and are 
estimated from MR (magnetic resonance) images (indirect observations)~\cite{hogea2008brain}. The estimated coefficients can be used to predict future tumor growth for planning treatment. As another example, in CO${}_2$ sequestration projects, the migration of injected CO${}_2$ into geological formations is governed by unknown fluid dynamics, where reservoir engineers can only collect sparsely distributed data~(direct observations) from wells, which are insufficient for learning the hidden fluid dynamics. However, we can easily collect waveform data~(indirect observation) whose evolution is determined by the fluid dynamics. It is desirable to use the waveform data to conduct high fidelity inverse modeling to learn the hidden fluid dynamics. With the learned dynamics, reservoir engineers will be able to answer whether there are or will be any leaks~\cite{leung2014overview,gcep2008,gcep_2009}.

\begin{figure}[hbtp]
    \centering
  \includegraphics[width=0.8\textwidth]{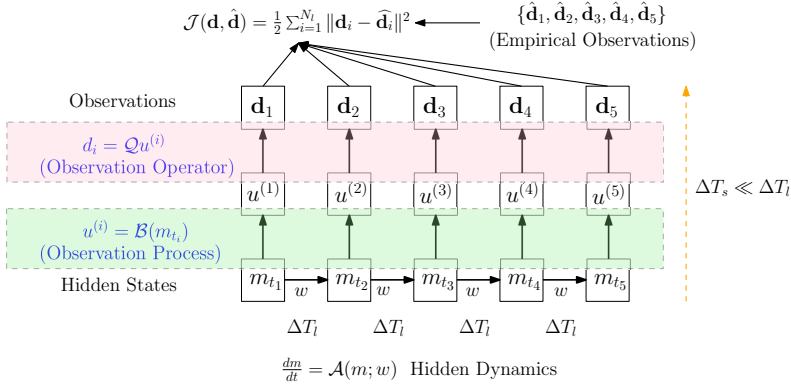}
  \caption{The hidden states such as concentration, pressure and saturation are not directly observable. $w$ are the unknown parameters to be determined through the PDE constrained optimization problem. The hidden states are coupled with another physical process called observation process. We collect data $\mathbf{d}_i = \mathcal{Q}u^{(i)}$ from receivers.}
  \label{fig:obs}
\end{figure}

In the general form, learning hidden dynamics from indirect observations can be formulated as a mathematical optimization problem involving systems of partial differential equations (PDE) 
\begin{equation}\label{equ:J}
  \begin{aligned}
	& \min_{w} L(w) = \mathcal{J}(\mathbf{d}, \hat{\mathbf{d}}) 
	= \frac{1}{2} \sum_{i=1}^{N_l} \|\mathbf{d}_i - \widehat{\mathbf{d}}_i\|^2 \\
	\mathrm{s.t.} \quad & \mathbf{d}_i(\bx, t) = \mathcal{Q}(u^{(i)}(\bx, t)) \\
    & u^{(i)}(\bx, t) = \mathcal{B}(m(\bx, t)) & \bx \in \Omega, \quad t_i \leq t \leq t_i + \Delta T_s \\	
    & \frac{d m(\bx, t)}{d t} = \mathcal{A}(m(\bx, t);w),\quad \mathcal{A}_{\mathrm{bc}}(m(\bx,t))=0 & 0 \leq t \leq T_l
\end{aligned}
\end{equation}
Here $\Omega$ is the computational domain and we identify three components (see \Cref{fig:obs} for a schematic description)
\begin{enumerate}
    \item \textbf{Hidden dynamics}. The hidden states $m(\bx, t)$, such as the fluid properties or the tumor concentration in the previous examples, are subject to some governing equations $\frac{dm}{dt}=\mathcal{A}(m;w)$ in the time horizon $[0, T_l]$. $\mathcal{A}$ is a differential operator (e.g., the reaction-advection-diffusion equation) and $w$ are unknown physical parameters. $\mathcal{A}_{\mathrm{bc}}(m(\bx,t))=0$ is the boundary condition.
    \item \textbf{Observation process}. Another physical process $\mathcal{B}$, such as wave propagation (or magnetic resonance imaging), is coupled with $m$ and produces seismic signals~(or images) $u = \mathcal{B}(m)$. Its time duration $[t_i, t_i+\Delta T_s]$ is much shorter than that of the hidden dynamics, $T_l$, and therefore we assume that $m$ is fixed during $t\in [t_i, t_i+\Delta T_s]$: $m(\bx, t) \approx m(\bx, t_i) := m_{t_i}(\bx)$.
    \item\textbf{ Observation operator}. In practice, only parts of the seismic signals~(or images) $u^{(i)}$ are observable because of the receiver constraints. This constraint is expressed by $\mathbf{d}_i = \mathcal{Q}(u^{(i)})$, where $\mathbf{d}_i$ is the data collected by the receivers. 
\end{enumerate}

When learning unknown parameters $w$ in hidden dynamics, gradient-based optimization methods for minimizing $L(w)$ have been proven to be effective and scalable (\cite{tromp2005seismic,oberai2003solution,martins2002coupled,bonnet2005inverse}). To obtain the gradients for the mathematical optimization, there are usually two prevailing methods: (1) state adjoint methods, where the gradients are derived and implemented by hand, which is a time-consuming and error-prone process; (2) automatic differentiation, where gradients are derived by storing and analyzing all the traces of forward simulation, which limits our ability to leverage memory and computation efficient algorithms for bottleneck parts. Combing the best of the two methods --- getting rid of the necessity of deriving gradients by hand while retaining the flexibility of incorporating custom built adjoint state codes --- is essential for building convoluted models to learn hidden dynamics. 

We propose a new framework, called intelligent automatic differentiation (IAD), for implementing algorithms for learning hidden dynamics described by PDEs from indirect data in general. Particularly, we build our framework based on the modern machine learning tool \texttt{TensorFlow}, which offers automatic differentiation \cite{gockenbach2002automatic,gremse2016gpu,giles2005using,courty2003reverse} functionality,  scalability and parallelism for our scientific computing problems. However, direct implementation of numerical schemes in \texttt{TensorFlow} is inefficient in three situations: (1) we need to write vectorized codes for efficient linear algebra computation, which can be difficult; (2) automatic differentiation consumes much memory; (3) there are more efficient implementation of gradient computation for certain parts of computation. In IAD, we deal with those challenges by incorporating custom operators.

In this paragraph, we describe how automatic differentiation works and how custom operators are implemented in IAD. When we implement the forward simulation, each operation is represented by a node in a computational graph. For example, the operation can be one time step of forward simulation for the hidden dynamics or the observation operator $\mathcal{Q}$. Those nodes are connected by directed edges that represent data and the arrow indicates the flow of the data~(\Cref{fig:cgraph}). To compute the gradient of $\frac{\partial L}{\partial w}$ in \Cref{fig:cgraph}, each operation in the path $w\leadsto L$ is equipped with two functions
\begin{equation}\label{equ:fb}
  \begin{aligned}
    y &= \texttt{forward}(x) \\
    \frac{\partial L}{\partial x}&= \texttt{backward}\left( \frac{\partial L}{\partial y} \right)
\end{aligned}
\end{equation}
The result in the second equation is equal to $\frac{\partial L}{\partial y}\frac{\partial y}{\partial x}$ according to the chain rule. \revise{\texttt{backward} can be interpreted as ``back-propagating''} the gradients from $\frac{\partial L}{\partial y}$ to $\frac{\partial L}{\partial x}$. By chaining together these operators we are able to obtain the gradients $\frac{\partial L}{\partial w}$ no matter how complicated the computational graph is. The decoupling of individual operations and systematic computation has other advantages. If we are trying out a new sub-model, e.g., replacing the advection-diffusion equation by the fractional diffusion equation for the hidden dynamics $\mathcal{A}$, we only need to substitute the corresponding operator and the rest of the graph remains the same. To implement custom operators, we manually implement both \texttt{forward} and \texttt{backward} in \cref{equ:fb} and insert this new operator into the original computational graph. 

\begin{figure}[htpb]
\centering
  \includegraphics[width=0.8\textwidth]{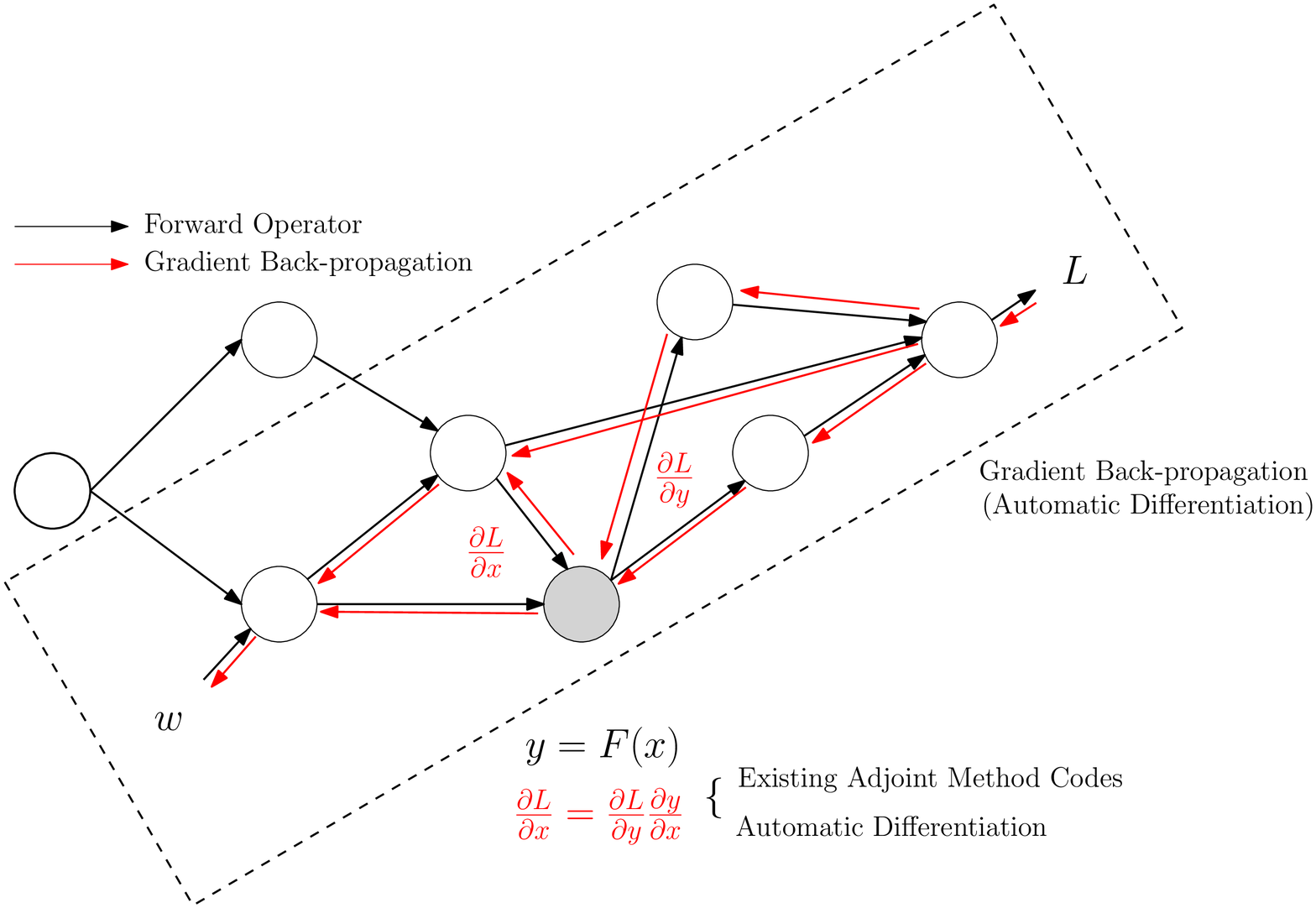}
  \caption{The overall simulation consists of multiple operators. For operators in the path $w\leadsto L$, we require that they can back-propagate the gradients $\frac{\partial L}{\partial y}$ to $\frac{\partial L}{\partial x}$. Automatic differentiation or custom built adjoint method codes can be used for this purpose. When all the operators are linked together, the system can output the gradient $\frac{\partial L}{\partial w}$ by passing the gradients through each individual operator.}
  \label{fig:cgraph}
\end{figure}
%
%

\begin{figure}[htpb]
\centering
  \includegraphics[width=0.8\textwidth]{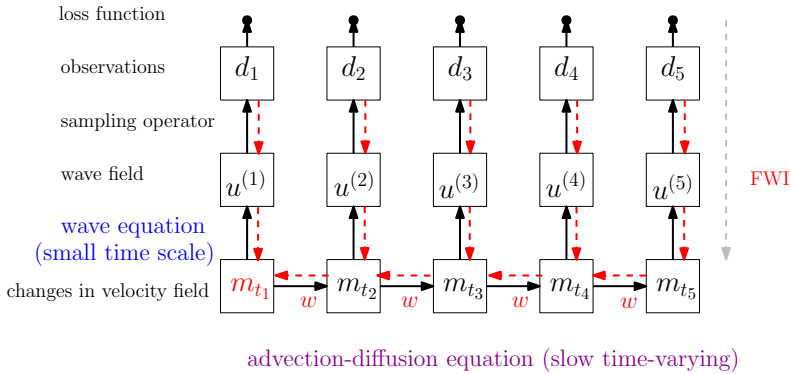}
  \caption{A paradigm for inverse problems with dynamical systems exhibiting multiple time-scales. The black arrows denote the data flow of forward simulation while the red arrows are for gradient back-propagation.}
  \label{fig:backward}
\end{figure}

Although this framework is applicable to a broader class of problems for learning hidden dynamics,  we focus on a  time-lapse monitoring problem in seismic imaging for concreteness. The problem can be described as follows (see \Cref{fig:backward} for a schematic description)
    \begin{multline*}
    \underbrace{\mbox{Medium Property}\xrightarrow{\mbox{Rock Physics Model}} \mbox{Acoustic Wave Velocity}}_{\mbox{This defines the hidden dynamics}\  \mathcal{A}} \\\xrightarrow{\mathcal{B}} \mbox{Wave Field} \xrightarrow{\mathcal{Q}} \mbox{Waveform Data}
\end{multline*}

In this problem, the hidden dynamics $\mathcal{A}$ describes the evolution of the acoustic wave velocity $v = m + m_{\mathrm{base}}$. The medium property~(e.g., concentration, saturation or pressure) determines the acoustic wave velocity, denoted by $v=m + m_{\mathrm{base}}$, where $m_{\mathrm{base}}$ is a reference velocity. The evolution of $v$ is determined by the flow dynamics and the rock physics model. For example, we may use an advection-diffusion equation to describe the evolution of concentration. The changes in the concentration of fluids alter the rock acoustic wave velocity because the fluids takes up pore space inside. This relationship is known as the rock physics model. We can obtain these models from experiments or first principles~\cite{mavko2009rock}. In practice, the interaction of medium property dynamics and the rock physics model can be quite complicated and the direct modeling of the acoustic wave velocity simplifies the procedure in the presence of sufficient data. As a prototype, we model the overall dynamics of $m$ with (fractional) advection-diffusion equations. In general, the model can be more sophisticated and tailored to specific applications, such as coupling a two-phase flow model and Gassmann's equations for rock physics. The observation process $\mathcal{B}$ describes the propagating wave $u^{(i)}$ in the rocks and the observation operator $\mathcal{Q}$ is the signal $u^{(i)}$ received at several locations. The acoustic wave velocity $m_{t_i}$ is not directly measurable, but we can measure the wave-field $u^{(i)}$ at the receiver location. The governing equation of this process is usually a wave equation, whose bulk modulus is a function of the acoustic wave velocity. 

We will also investigate the stability of gradient computation for different hidden dynamics, including the fractional advection-diffusion equation. We show that if we use an implicit scheme for the forward simulation, under mild assumptions, our method enjoys certain stability. The analysis is also true for the time-fractional partial differential equation.

The rest of the paper is organized as follows. We introduce the background and the mathematical setting of the model problem in \Cref{sect:ps}. \Cref{sect:haad} describes the proposed method and implementation details. In \Cref{sect:stab}, the stability is analyzed. The model problem is solved numerically and discussed in \Cref{sect:ap}. Finally, we summarize and point out limitations and future research opportunities in \Cref{sect:conc}.

\section{Problem Setup}\label{sect:ps}

\subsection{Background}

In time-lapse monitoring problems, certain governing equations control the evolution of medium properties. The parameters in the governing equations cannot be observed directly. Instead, elastic or acoustic waves can be excited to sense those properties. Repeated surveys of such kind using wave phenomena or the so-called full-waveform inversion (FWI) can help to reconstruct the hidden parameters in the governing equations, and may provide insights for future predictions. FWI is currently the state-of-the-art method in seismic inversion~\cite{tarantola1984inversion,virieux2009overview,pratt1999seismic,biondi2012tomographic}. It computes the gradient of the misfit between estimated and observed waveform data with respect to model parameters. 

The slowly evolving processes are governed by certain PDEs, i.e.,\ $\frac{d m}{d t}=\mathcal{A}(m;w)$ in \Cref{equ:J}. The task here is to use waveform information $\hat {\mathbf{d}}_i$ to infer the hidden parameters $w$ in those PDEs. Coupled dynamic systems in different time scales are involved in this inverse problem. It can be assumed that during the ``fast'' wave propagation time scale, the ``slow-time'' property changes are negligible.

\subsection{Hidden Dynamics $\frac{d m}{d t}=\mathcal{A}(m;w)$}\label{sect:model problem}

We consider three models for describing the slow time-varying process~(hidden dynamics) in \Cref{equ:J}. These models are representative for describing the dynamics of the fluid properties. 

\begin{itemize}	
\item The advection diffusion equation.

In this case, the configuration within which the diffusion takes place are also moving in a preferential direction. The governing equation can be mathematically described as 
\begin{equation}\label{equ:ad}
\frac{{\partial m}}{{\partial t}} = {b_1}\frac{{\partial m}}{{\partial x}} + {b_2}\frac{{\partial m}}{{\partial y}} + a\left( {\frac{{{\partial ^2}m}}{{\partial {x^2}}} + \frac{{{\partial ^2}m}}{{\partial {y^2}}}} \right)
\end{equation}
with the initial condition and Dirichlet boundary condition
\begin{equation}\label{equ:ib}
m(\bx, 0; w) = m_0(\bx;w) \qquad 	m(\bx, t; w) = 0 \qquad \bx\in \partial\Omega
\end{equation}

\item Time-fractional diffusion equation.

The time-fractional PDEs are very useful tools but computational challenging. They have been used for describing memory and hereditary properties of many materials. However, solving inverse problems in the time-fractional PDEs exhibits unique challenges because in the discrete adjoint method, solutions from all previous steps are coupled together~\cite{maryshev2013adjoint}. This not only adds complexity to algorithms and implementation but also places pressure on memory and computation. Nevertheless, the inverse modeling for the fractional PDEs is substantially simplified by using our method. The gradient computation in the multi-step scheme is properly handled by \texttt{ADCME} and requires no implementation efforts except for the forward simulation.

The time-fractional advection diffusion equation is 
\begin{equation}\label{equ:tad}
	{}_0^CD_t^\alpha m = a\left( {\frac{{{\partial ^2}m}}{{\partial {x^2}}} + \frac{{{\partial ^2}m}}{{\partial {y^2}}}} \right)
\end{equation}
where ${}_0^CD_t^\alpha$ is the Caputo derivative with index $\alpha\in (0,1)$
\begin{equation}
	{}_0^CD_t^\alpha f(t) = \frac{1}{\Gamma(1-\alpha)}\int_0^t \frac{f'(\tau)d\tau}{(t-\tau)^\alpha}
\end{equation}
The initial and boundary conditions are the same as \Cref{equ:ib}.

\item Space-fractional diffusion equation

The space-fractional diffusion equation is
\begin{equation}\label{equ:fl}
	\frac{{\partial m}}{{\partial t}} =  - a(-\Delta)^s m
\end{equation}
where initial and boundary conditions are the same as \Cref{equ:ib}, and $(-\Delta)^s$ is the fractional Laplacian, which describes long distance interactions~\cite{vazquez2017mathematical}. Let p.v.\ denote the principal value integration, the fractional Laplacian $(-\Delta)^s$, indexed by $s\in (0,1)$, 
\begin{equation}\label{equ:sad}
	(-\Delta)^s m(\bx) = c_{d,s} \mathrm{p.v.}\int_{\RR^d}\frac{m(\bx)-m(\by)}{|\bx-\by|^{d+2s}}d\by, \quad c_{d,s} = \frac{2^{2s}\Gamma\left( s+\frac{d}{2} \right)}{\pi^{d/2}\left|\Gamma(-s) \right|}
\end{equation}

Discretization of the fractional Laplacian results in a dense matrix, which significantly increase the memory and computation consumption. To alleviate this problem, we consider a square domain with homogeneous Dirichlet boundary condition and apply the Fourier spectral method. In this case, to avoid nonzero parts of $m$ penetrating the boundary, we assume the transport coefficients $b_1=b_2=0$. 
\end{itemize}

The hidden dynamics can be more sophisticated and domain-specific, such as the black-oil equations in reservoir simulation~\cite{chen2000formulations} or the PDE for tumor growth modeling~\cite{knopoff2013adjoint}. For numerical stability, we resort to implicit schemes for the numerical simulation, although they are usually more computational expensive. We show how we discretize the equation in \Cref{sect:d_ad}.

\subsection{Observation Process $m_{t_i} = \mathcal{B}(u^{(i)})$ and Operator $\mathcal{Q}$}
The second constraint $m_{t_i} = \mathcal{B}(u^{(i)})$ in \Cref{equ:J} is the acoustic wave equation in the stress-velocity formulation:
\begin{equation}
    \label{eq:acoustic_eqn}
	\begin{cases}
	    \partial_t p(\mathbf{x},t) = -K(\mathbf{x},t; {m}(\bx, t))
	    \; \nabla \cdot \mathbf{v}(\mathbf{x},t) + s(\mathbf{x},t), \\
	    \rho(\mathbf{x},t)\; \partial_t \mathbf{v}(\mathbf{x},t) = -\nabla p(\mathbf{x},t), \\
	    p(\mathbf{x}, t) = 0, \quad \bv(\mathbf{x}, t) = \mathbf{0} \quad \mathbf{x} \in \partial \Omega \\
	    p(\mathbf{x}, t) = 0, \quad \mathbf{v}(\mathbf{x}, t) = \mathbf{0}, \quad t\leq0, \quad \mathbf{x} \in \Omega,
    \end{cases}
\end{equation}
The computational domain is $\partial \Omega \times [0, \Delta T_s]$. Compared to the hidden dynamics, \Cref{eq:acoustic_eqn} occurs in a relatively short period, that is $\Delta T_s \ll \Delta T_l$. Therefore, we can view ${m}(\bx, t)$ as fixed and approximately equal to
\begin{equation}
	m(\bx, t) \approx m(\bx, t_i) := m_{t_i}(\bx)
\end{equation}
where $K(\mathbf{x},t; {m}(\bx, t))=\big({m}(\mathbf{x},t)+{m}_{\mathrm{base}}(\mathbf{x},t)\big)^2\rho(\mathbf{x},t)\approx \big(m_{t_i}+{m}_{\mathrm{base}}(\mathbf{x},t)\big)^2\rho(\mathbf{x},t)$  is the bulk modulus for a given baseline quantity ${m}_{\mathrm{base}}(\bx, t)$, and $\rho$ is the density, $\bv(\bx,t)=(u(\bx,t),v(\bx,t))$ is the particle velocity and $s(\mathbf{x},t)$ is the source function. In our case, we consider ${m}_{\mathrm{base}}(\bx, t)$ and $\rho(\mathbf{x},t)$ as constant, and they are independent of $\bx$ and $\bt$. The acoustic equation is reduced from the elastic wave equation, and $p$ is the pressure. The waveform is given by
\begin{equation}
	\mathbf{u}^{(i)}(\mathbf{x},t) = \begin{bmatrix}
		\mathbf{v}(\mathbf{x},t) & p(\mathbf{x},t)
	\end{bmatrix}
\end{equation}

The observation operator $\mathcal{Q}$ are projections of the pressure $p(\mathbf{x},t)$ and particle velocity $\mathbf{v}(\mathbf{x},t)$ onto the temporal-spatial space of receivers~(\Cref{fig:rec}), i.e.,
\begin{equation}
	\mathcal{Q}\mathbf{u}^{(i)} = p(\mathbf{x}_k,t)
\end{equation}
for all receiver locations $\{\mathbf{x}_k\}$. The model parameter, $m_{t_i}$, has very small changes during one seismic survey on the scale of hours or days, and evolves slowly on a time scale of months or even years.

\begin{figure}[htpb]
\centering
  \includegraphics[width=0.5\textwidth]{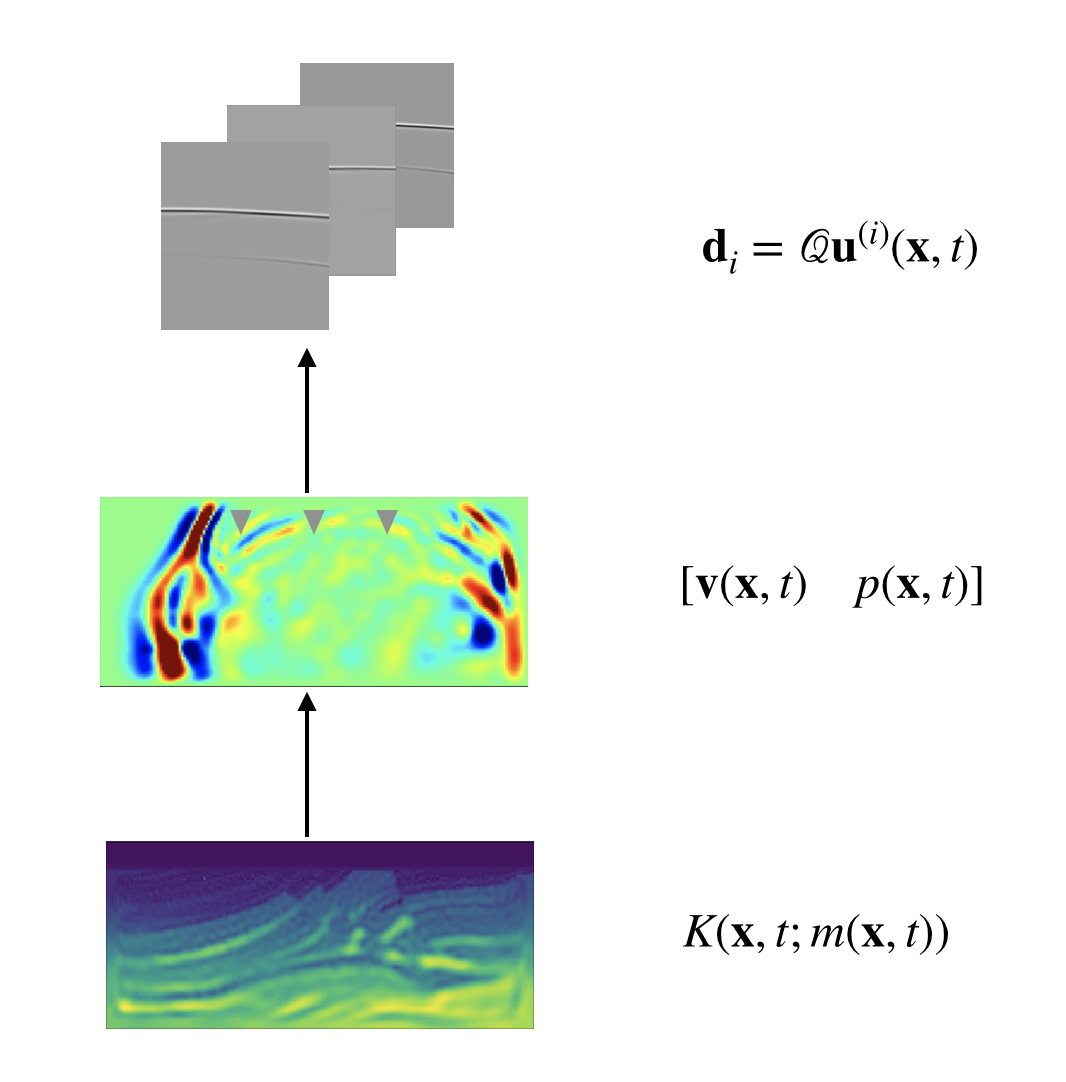}
  \caption{The observation operator $\mathcal{Q}$ are projections of pressure $p(\mathbf{x},t)$ onto the temporal-spatial space of receivers $p(\mathbf{x}_k, t)$.}
  \label{fig:rec}
\end{figure}

Note that the acoustic equation \Cref{eq:acoustic_eqn} has a free-surface boundary condition. Due to the relative short propagation time, waves can only reach the free-surface on part of the boundary. To guarantee that no waves come back from the rest part of the boundary, absorbing boundary conditions are imposed. In this research, we use the convolutional perfectly match layers~(CPML)~\cite{li2010convolutional}, whose implementation is discussed in detail in the discretization section. The CPML works by substituting the spatial derivative using 
\begin{equation}\label{eq:cpml_derivative}
   {\partial _{\tilde x}}c = \left( {\frac{1}{{{\kappa _x}}}{\partial _x} + {\xi _x} \star {\partial _x}} \right)c = \frac{{{\partial _x}c}}{{{\kappa _x}}} - {\eta _x}\int_0^t {{e^{ - {\alpha _x}\tau }}{\partial _x}c(x,t - \tau )d\tau } ,
\end{equation}
where ${\xi _x}(t) =  - ({\eta _x})H(t)\exp ({\alpha _x}t)$, ${\alpha _x} = ({d_x}/{\kappa _x} + {a^x})$, $\eta_x = d_x/\kappa_x^2$, $H(t)$ is a Heaviside function, and $\star$ is convolution in time. The damping profile $d_x(x)$, and linearly varying parameters $\kappa_x$ and $a_x$ can be found in \cite{martin2009unsplit}. Here where $c$ can be $u$, $v$ or $p$.

\subsection{Discretization of the (Fractional) Advection Diffusion Equation}\label{sect:d_ad}

For numerical stability, we use an implicit method to discretize the (fractional) advection diffusion equations in \Cref{sect:model problem}. The interval $[0,T_l]$ is split into $n_T$ equal length intervals $0=\tau_0 < \tau_1 < \ldots < \tau_{n_T} = T$ with $\tau_{i+1}-\tau_i = \Delta \tau$. We choose $\Delta \tau$ so that $t_i$ in \Cref{equ:J} coincides with some $\tau_j$. We consider uniform discretization of $\Omega=[0,L]^2$ with step size $h$, with grid points $\bx_{ij} = (x_i,x_j)$, $i,j=0,1,2,\ldots,N$. $m_{ij}^n$ is the numerical approximation to $m(\bx_{ij}, n\Delta \tau; w)$. 

\begin{itemize}
	\item The advection diffusion equation.
\begin{equation}\label{equ:dis_ad}
	\begin{aligned}
	 \frac{{m_{i,j}^{n + 1} - m_{i,j}^n}}{{\Delta \tau}} =&   {b_1}\frac{{m_{i + 1,j}^{n+1} - m_{i - 1,j}^{n+1}}}{{2h}} + {b_2}\frac{{m_{i,j + 1}^{n+1} - m_{i,j - 1}^{n+1}}}{{2h}} \\
	& + a\frac{{m_{i - 1,j}^{n+1} + m_{i + 1,j}^{n+1} + m_{i,j - 1}^{n+1} 
	+ m_{i,j + 1}^{n+1} - 4m_{i,j}^{n+1}}}{{{h^2}}}	
	\end{aligned}	
\end{equation}
The update scheme can be written as~(after taking into account of the boundary condition)
\begin{equation}\label{equ:scheme0}
	(I + \Delta \tau A) \bu^{n+1} = \bu^n
\end{equation}
where $\bu^n = \begin{bmatrix}
	m_{1,1}^{n + 1} & m_{2,1}^{n + 1} & \cdots & m_{N-1,1}^{n + 1} & m_{1,2}^{n + 1} & \cdots & m_{N-1,N-1}^{n + 1}
\end{bmatrix}$ and $-A$ is the discrete advection diffusion operator, independent of $\Delta \tau$. Since we have used an implicit scheme, the update scheme is unconditional stable~\cite{thomas2013numerical}, i.e., 
$$\rho((I+\Delta \tau A)^{-1})<1,\ \forall \Delta \tau>0$$

\item The time-fractional diffusion equation.

Let 
\begin{equation}
	G_m = (m+1)^{1-\alpha} - m^{1-\alpha}, \quad m\geq 0, \quad 0<\alpha<1
\end{equation}
then the fractional derivative can be discretized as~\cite{jiang2017fast}
\begin{equation*}
\begin{aligned}
& {}_0^CD_\tau^\alpha u(\tau_n) \\
 = &\frac{\Delta \tau^{-\alpha}}{\Gamma(2-\alpha)}\left[G_0 u_n - \sum_{k=1}^{n-1}(G_{n-k-1}-G_{n-k})u_k + G_n u_0 \right] + \mathcal{O}(\Delta \tau^{2-\alpha})	
\end{aligned}
\end{equation*}
therefore the update scheme is 
\begin{equation}\label{equ:tscheme}
	(I + \Gamma(2-\alpha)\Delta\tau^\alpha A)\bu^{n+1} = \Gamma(2-\alpha)\Delta\tau^\alpha \bu^n + \sum_{k=1}^{n-1}(G_{n-k-1}-G_{n-k})\bu^k - G_n \bu^0
\end{equation}

\item The space-fractional diffusion equation.

We can use fast Fourier transform to solve \Cref{equ:fl}. Let $\hat \bu^n$ be the discrete Fourier transform of $\bu^n$, then we have
\begin{equation}
	(1+\Delta \tau |\bm{\xi}|^{2s})\hat \bu^{n+1} = \hat \bu^n
\end{equation}
and therefore
\begin{equation}
	 \bu^{n+1} = \mathcal{F}^{-1}\left( \frac{\hat\bu^n}{1+\Delta \tau|\bm{\xi}|^{2s}} \right)
\end{equation}
where $\mathcal{F}^{-1}$ denotes the inverse discrete Fourier transform. The scheme is equivalent to 
\begin{equation}\label{equ:fscheme}
	(I + \Delta \tau A^s)\bu^{n+1} = \bu^n
\end{equation}
for an appropriate operator $A^s$ that depends on $s$ and $\rho((I+\Delta \tau A^s)^{-1})<1$.

\end{itemize}

\section{Intelligent Automatic Differentiation}\label{sect:haad}

In this section, we describe the algorithm underlying \texttt{ADCME}. As depicted in \Cref{fig:cgraph}, the individual operators are compiled together in an acyclical computational graph with an automatic differentiation framework, which enables automatic gradient back-propagation. For each operator $y = F(x)$, if the operator is in the path between the unknown parameter $w$ and final output $L$, it requires a gradient transformation feature, i.e., \texttt{backward} function, that can compute $\frac{\partial L}{\partial x}$ given $\frac{\partial L}{\partial y}$. There are two ways to implement this \texttt{backward} function in \texttt{ADCME}: (1) automatic differentiation, and (2) custom built adjoint method codes. In what follows, we briefly introduce reverse model automatic differentiation and then present how our \texttt{backward} function is implemented in custom operators.

\subsection{Automatic Differentiation}

We use a gradient-based method to minimize the objective function \Cref{equ:J}. We can derive and implement the gradients directly by adjoint state methods. However, as the system grows more complex, it requires tremendous effort for derivation, implementation and debugging. We tackle the problem by developing \texttt{ADCME}, which uses the automatic differentiation capacity of the modern machine learning tool, e.g., \texttt{TensorFlow}~\cite{girija2016tensorflow}. 

\texttt{ADCME} provides us a way of implementing numerical schemes and computing the gradients 
$\frac{\partial \mathcal{J}}{\partial w}$ using reverse mode automatic differentiation~(AD)~\cite{baydin2018automatic}.  AD applies symbolic differentiation at the elementary operation level. In AD, all numerical computations are ultimately compositions of a finite set of elementary operations for which derivatives are known, and combining the derivatives of the constituent operations through the chain rule yields the derivative of the overall composition. In a nutshell, researchers do not need to worry about deriving and implements derivatives. Most of the differentiation and optimization processes are taken care of by the software. \texttt{ADCME} also allows users to implement custom forward simulation and gradients operators, which is quite useful for incorporating custom built codes without much modification.

\subsection{Incorporating Custom Built Adjoint Method Codes via Custom Operators}

A naive implementation in \texttt{ADCME} with the \texttt{TensorFlow} backend without code vectorization is inefficient and infeasible for large-scale problem. Even when the code is vectorized, we may not be able to compute the gradients due to memory or computation constraints. One solution is to implement a custom operator, which we have used for the FWI component and one-step forward simulation of advection-diffusion equation. 

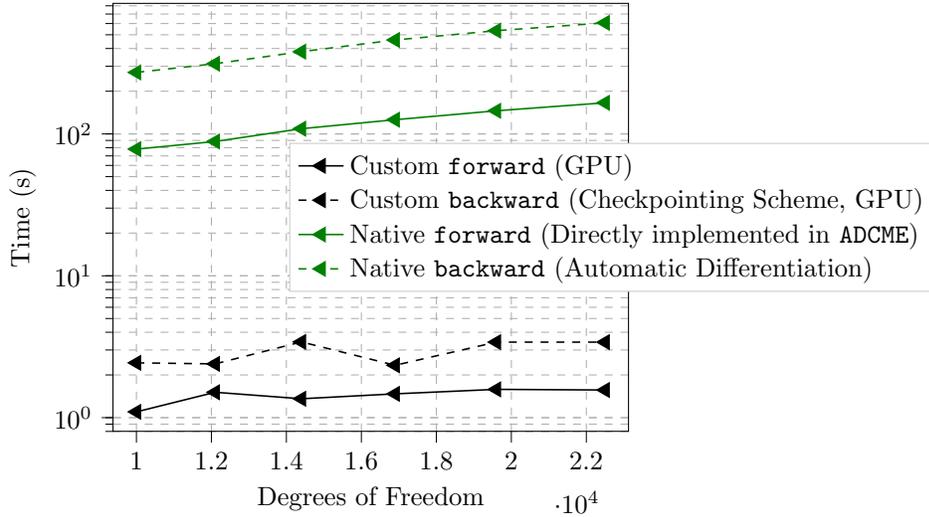
\begin{figure}[htpb]
\centering
\scalebox{1.0}{
\begin{tikzpicture}

\begin{axis}[
legend cell align={left},
legend style={at={(1.6,0.5)}, anchor=east, draw=white!80.0!black},
log basis y={10},
tick align=outside,
tick pos=left,
grid = both,
grid style={dashed,gray},
x grid style={white!69.01960784313725!black},
xlabel={Degrees of Freedom},
xmin=9375, xmax=23125,
xtick style={color=black},
y grid style={white!69.01960784313725!black},
ylabel={Time (s)},
ymin=0.801147995202948, ymax=832.290727341659,
ymode=log,
ytick style={color=black},
]
\addplot [semithick, black, mark=triangle*, mark size=3, mark options={solid,rotate=90}]
table {%
10000 1.098573753
12100 1.510011265
14400 1.358987617
16900 1.471672724
19600 1.583162684
22500 1.566282273
};
\addlegendentry{Custom \texttt{forward} (GPU)}
\addplot [semithick, black, dashed, mark=triangle*, mark size=3, mark options={solid,rotate=90}]
table {%
10000 2.433655199
12100 2.393301131
14400 3.419066701
16900 2.339784587
19600 3.407999645
22500 3.408061813
};
\addlegendentry{Custom \texttt{backward} (Checkpointing Scheme, GPU)}
\addplot [semithick, green!50.0!black, mark=triangle*, mark size=3, mark options={solid,rotate=90}]
table {%
10000 78.141974648
12100 88.471512506
14400 108.818969224
16900 125.963000824
19600 145.582793621
22500 165.79505068
};
\addlegendentry{Native \texttt{forward}  (Directly implemented in \texttt{ADCME})}
\addplot [semithick, green!50.0!black, dashed, mark=triangle*, mark size=3, mark options={solid,rotate=90}]
table {%
10000 271.406630611
12100 312.662992337
14400 380.114894473
16900 459.434523054
19600 534.385552022
22500 606.957926871
};
\addlegendentry{Native \texttt{backward} (Automatic Differentiation)}
\end{axis}

\end{tikzpicture}}
\caption{Comparison of time consumed for the optimized operator and its automatic differentiation counterparts. Compared to the native automatic differentiation implementation~(\Cref{sect:code}), the custom operator enjoys more than 100 times acceleration and scales better with respect to the degrees of freedom.}
\label{fig:benchmark}
\end{figure}

It is crucial to allow for this kind of flexibility. In our case the FWI operation is the most time consuming operator. We have implemented a heavily optimized version of FWI using CUDA. During the forward time stepping, unlike the usual automatic differentiation where all the intermediate data must be retained for computing gradients, we only save information around the boundary. This information is then used to reconstruct the intermediate data during back-propagation. Thus we have saved the scarce memory resources on GPU and take advantage of the high FLOPS in the parallel computing. The process can also be viewed as a special form of checkpointing scheme~\cite{charpentier2001checkpointing}. For example, we compare the time consumed for the simulation of wave propagation and its gradient computation given different degrees of freedoms. The result is shown in \Cref{fig:benchmark} for 8 sources and 8 receivers and we observe more than 100 times acceleration for an optimized operator~(the manually implemented adjoint method) compared to its automatic differentiation counterparts on CPUs. For detailed setup for numerical computations, see \Cref{sect:ap}.

\subsection{Example: Implementing \texttt{backward} for \Cref{equ:J}}

First we describe the strategies for implementing the gradients for the objective function subject to the constraints in \Cref{equ:J} using automatic differentiation. We used the \textit{discretize-then-optimize} approach. We first discretize the wave equation and the advection-diffusion equation as discussed in \Cref{sect:d_ad}. The advection-diffusion equation involves unknown parameters $w$. We assume the initial configuration $\mathbf{m}_{t_1}=\mathbf{m}_0$ is given, but it can also be an unknown variable as well. We use bold $\mathbf{m}$ since it is the discretized version of $m$. To make the discussion clearer, we first introduce the notation for partial derivatives: $\frac{\partial f}{\partial \bx}$ and $\frac{Df}{D\bx}$. The first notation treats $\bx$ as independent variable of $f$ and the second treats $\bx$ as dependent variable, i.e., $f$ has one or more arguments that depends on $\bx$. For example, for $z = f(x, y)$, if $y$ is a function of $x$, i.e., $y=y(x)$, then we have
\begin{align}
	\frac{\partial f(x,y)}{\partial x} &= f_1(x,y)\\
	\frac{Df(x,y)}{Dx} &= f_1(x,y) + f_2(x,y)y'(x)
\end{align}
here $f_1(x,y)=\lim_{\Delta x\rightarrow 0}\frac{f(x+\Delta x, y)-f(x,y)}{\Delta x}$, $f_2(x,y)=\lim_{\Delta y\rightarrow 0}\frac{f(x, y+\Delta y)-f(x,y)}{\Delta y}$.

Therefore, by the chain rule, when computing the gradients of $\mathcal{J}$ with respect to $w$, we have
\begin{align}
	\frac{{D \mathcal{J}}}{{D w}} =& \sum\limits_{i = 1}^{{N_l}} {\frac{{\partial \mathcal{J}}}{{\partial {\mathbf{m}_{{t_i}}}}}} \frac{{D {\mathbf{m}_{{t_i}}}}}{{D w}}\\
	\frac{{D {\mathbf{m}_{{t_i}}}}}{{D w}} = & \frac{{\partial{\mathbf{m}_{{t_i}}}}}{{\partial w}} + \frac{{\partial {\mathbf{m}_{{t_i}}}}}{{\partial {\mathbf{m}_{{t_{i - 1}}}}}}\frac{{D {\mathbf{m}_{{t_{i - 1}}}}}}{{D w}}
\end{align}

In our example, the gradients $\frac{{\partial \mathcal{J}}}{{\partial {\mathbf{m}_{{t_i}}}}}$ can be computed using the standard FWI method. The gradients $\frac{{\partial {\mathbf{m}_{{t_{k + 1}}}}}}{{\partial {\mathbf{m}_{{t_k}}}}}$ can also be computed using the adjoint state method. \revise{Technically, we only need to implement one step forward propagation \Cref{equ:dis_ad} and the associated gradient operator, and the overall simulation for the hidden dynamics involves repeatedly calling this operator.} Since for reverse mode automatic differentiation, there is no need to form the Jacobian explicit, we only need to implement the gradients of a functional, $\mathcal{J}$, with respect to ${\mathbf{m}_{{t_k}}^n}$, $w$, assuming $\frac{{\partial \mathcal{J}}}{{\partial \mathbf{m}_{{t_k}}^{n + 1}}}$ is known. Here the superscript $n$ in $\mathbf{m}_{t_k}^n$ denotes the time step between $t_k$, $t_{k+1}$ because the evolution from $\mathbf{m}_{t_k}$ to 
$\mathbf{m}_{t_{k+1}}$ is computed with a time-stepping method.
\begin{equation}\label{equ:gradJ}
    \begin{aligned}
	 & \frac{{\partial \mathcal{J}}}{{\partial \mathbf{m}_{{t_k}}^n}} = \frac{{\partial \mathcal{J}}}{{\partial \mathbf{m}_{{t_k}}^{n + 1}}}\frac{{\partial \mathbf{m}_{{t_k}}^{n + 1}}}{{\partial \mathbf{m}_{{t_k}}^n}}  \cr 
  & \frac{{\partial \mathcal{J}}}{{\partial w}} = \frac{{\partial \mathcal{J}}}{{\partial \mathbf{m}_{{t_k}}^{n + 1}}}\frac{{\partial \mathbf{m}_{{t_k}}^{n + 1}}}{{\partial w}}
\end{aligned}
\end{equation}

We remark that even for FWI and \Cref{equ:gradJ}, we can avoid implementing the gradients by using the automatic differentiation provided by \texttt{TensorFlow}. \revise{However, more often than not, we need to write vectorized codes for efficiency if we use \texttt{TensorFlow} directly, which makes it cumbersome to implement some boundary conditions and apply special linear solvers~(such as algebraic multi-grid method). Instead, we can design custom operators and embed the \texttt{C++} codes into the workflow. As an example, consider implementing a differentiable operator that maps  $\mathbf{m}_{{t_k}}^n$ to $\mathbf{m}_{{t_k}}^{n+1}$, with \texttt{ADCME}, after generating the custom operator template using command \texttt{customop} in \texttt{ADCME}, the user implements two functions \texttt{forward} and \texttt{backward}}
\begin{align}
	\mathbf{m}_{{t_k}}^{n+1} &= \texttt{forward}(\mathbf{m}_{{t_k}}^{n}, w)\\
	\frac{{\partial \mathcal{J}}}{{\partial \mathbf{m}_{{t_k}}^n}},\ \frac{{\partial \mathcal{J}}}{{\partial w}} &= \texttt{backward}\left(\frac{{\partial \mathcal{J}}}{{\partial \mathbf{m}_{{t_k}}^{n + 1}}},\mathbf{m}_{{t_k}}^{n+1},  \mathbf{m}_{{t_k}}^{n}, w\right) 
\end{align}
the backward operator can be implemented using the chain rule.

The forward simulation operators and gradients are eventually integrated by \texttt{TensorFlow} automatically. It enables an easy-to-use and friendly interface for the user and one is free to solve the optimization problem \Cref{equ:J} using any gradient based optimizers, such as gradient descent methods, \texttt{L-BFGS-B}, etc.

\subsection{While Loops}

A critical ingredient for the dynamical system is to avoid writing explicit loops, which results in large computational graphs according to reverse mode automatic differentiation. Fortunately, \texttt{TensorFlow} provides a very powerful tool to do this kind of tasks: \texttt{while\_loop}. Instead of creating new subgraphs for every iteration, \texttt{TensorFlow} introduces a single special control flow structure to the computational graph. For technical details, we refer the readers to the paper \cite{whitepap56:online}. In this way, the loops can be done efficiently. We show such an example in \Cref{sect:code}.

\section{Stability Analysis}\label{sect:stab}

In this section, we consider the stability for the gradient computation of hidden dynamics with memory. For example, the time-fractional partial differential equation falls into this category. Given a dynamic system, assume that the state function $u(\bx, t)$ is discretized at $0=t_0\leq t_1 \leq \ldots \leq t_n = T$ with $t_{i+1}-t_i=\Delta t$, $n\Delta t = T$ and spatial locations $\{\bx_i\}_{i=1}^M$, $u^i_j$ is the approximation to $u(\bx_i, t_j)$. The discrete numerical scheme for computing $\bu^i = \{u^i_j\}$ is 
\begin{equation}\label{equ:Fi}
  \begin{aligned}
	F_1(\bu^1, \bu^0, \bt; \Delta t) &=0 \\
	F_2(\bu^2, \bu^1, \bu^0, \bt; \Delta t) &=0 \\
	\ldots \\
	F_n(\bu^n, \bu^{n-1},\ldots, \bu^0, \bt; \Delta t) &=0 
\end{aligned}
\end{equation}
where $F_i:\RR^M\rightarrow \RR^M$ can be linear or nonlinear in $\bu^0$, $\bu^1$, $\ldots$, $\bu^i$ and $\bt$ is the unknown parameter. $\bu^i$ can be obtained by solving the $i$-th equation. We assume a multi-step implicit form so that it allows a broad class of numerical schemes.

The quantity of interest, $J$, usually the loss function, is an explicit function of $\bu^0$, $\bu^1$, $\ldots$, $\bu^n$
\begin{equation}
	J = J(\bu^0, \bu^1, \ldots, \bu^n)
\end{equation}
which depends on $\bt$ through $\bu^0$, $\bu^1$, $\ldots$, $\bu^n$.

We minimize $J$ through a first-order or quasi-second order method, which requires the gradients
\begin{equation}
	\frac{DJ}{D\bt} = \sum\limits_{i = 0}^n {\frac{{\partial J}}{{\partial {{{\bu}}^i}}}\frac{{D{{{\bu}}^i}}}{{D\bt}}} 
\end{equation}

\Cref{fig:J} shows an example where $F_i$ only depends on $\bu^i$, $\bu^{i-1}$, $\bt$ and $\Delta t$, i.e.,\ a single-step scheme. In this example, $\frac{\partial \bu^{n-1}}{\partial \bt}$ denotes a partial derivative where the independent variables are $\bt$ and all the previous $\bu^i$ ($i \le n-2$). $\frac{D \bu^{n-1}}{D \bt}$ denotes the derivative with respect to $\bt$ where now all the variables $\bu^i$ are assumed to be functions of $\bt$ (that is, they are considered dependent variables).

\begin{figure}[hbtp]
\centering
  \includegraphics[width =0.5\textwidth]{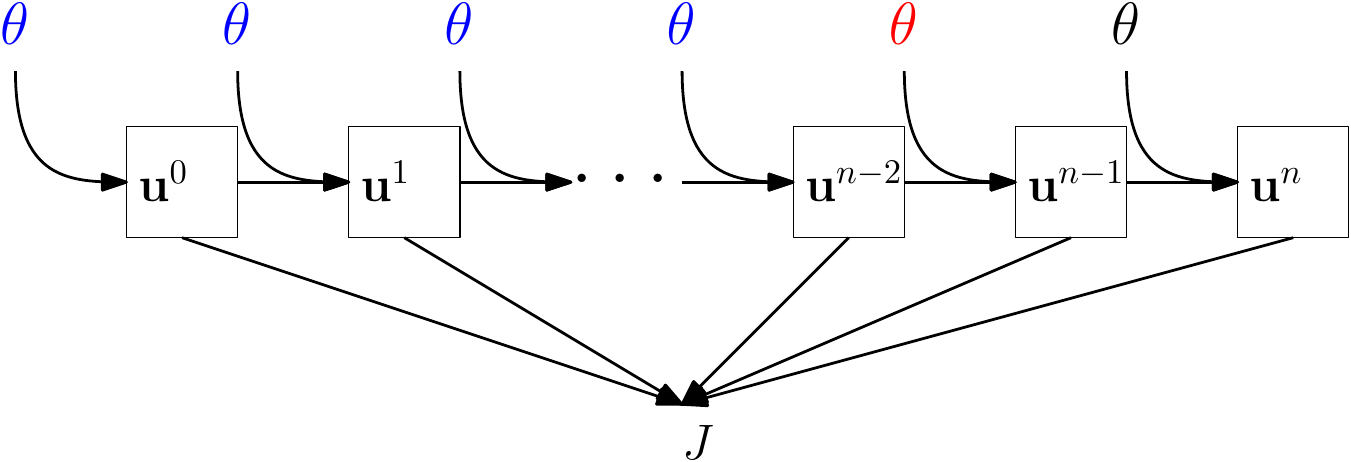}
  \caption{Schematic diagram for $\frac{D}{D\bt}$ and $\frac{\partial}{\partial\bt}$. In this example, when defining $\frac{\partial \bu^{n-1}}{\partial \bt}$ the independent variables are {$\bt$} and $\bu^i$, $i \le n-2$. While in $\frac{D \bu^{n-1}}{D \bt}$, the only independent variable is $\bt$ (all $\bu^i$, $i \le n-2$ are considered functions of $\bt$).}
  \label{fig:J}
\end{figure}

Now we can define the stability of the method. The proof can be found in \Cref{sect:proof}.
\begin{definition}[Stability]
	Consider the discrete numerical scheme \newline \Cref{equ:Fi} for approximating $u(\bx,t)$ for $t\in [0,T]$. We say that the numerical scheme is stable if there exists a constant, independent of $\Delta t$, such that
	\begin{equation}
		\left\|\frac{DJ}{D\bt}\right\| \leq C \left(\sum_{i=0}^n\left\| \frac{\partial J}{\partial \bu^i} \right\|^2\right)^{\frac{1}{2}}
	\end{equation}
	where $\|\cdot\|$ is 2-norm. 
\end{definition}

The definition requires that the gradients do not explode as $\Delta t\rightarrow 0$. Now we consider a special case where $F_i$, $J$ has the form
\begin{equation}\label{equ:Fiex}
\begin{aligned}
	F_i(\bu^i, \bu^{i-1}, \ldots, \bu^0, \bt; \Delta t) &=  A_i \bu^i -\sum_{j=0}^{i-1} a_{ij} \bu^{j} \\
	J &=  J(\bu^n)
\end{aligned}
\end{equation}
the operators $F_i$, $J$ in this paper have the form in \Cref{equ:Fiex}. 

\begin{theorem}\label{thm:main}
	Let $F_i$, $J$ have the form in \Cref{equ:Fiex}. Assume $C_1$, $C_2$, $C_0'$ are positive constants independent of $i$ and $\Delta t$,  $A_i$, $a_{ij}$ satisfy 
	\begin{enumerate}
		\item $\sum_{j=0}^{i-1} a_{ij} <1$, $a_{ij}\geq 0$, $\forall 0\leq j<i$;
		\item $\rho(A_i^{-1}) \leq 1- C_1(\Delta t)^\alpha$, where $\rho$ is the spectrum radius;
		\item $\left\| \frac{\partial\bu^i}{\partial\bt} \right\| \leq C_2(\Delta t)^\alpha$ for a constant $\gamma>0$.
		\item $\dfrac{\mathop {\max }\limits_{0 \leqslant i \leqslant n} \left\| {\frac{{\partial {{\bu}^i}}}{{\partial \bt}}} \right\|}{1-\mathop {\max }\limits_{0 \leqslant i \leqslant n} \rho (A_i^{ - 1})} \geq C_0'$
	\end{enumerate}
	here $\alpha\in (0,1]$; then the numerical scheme is stable. 
\end{theorem}

\begin{remark}
	The first three assumptions are easy to understand, and they are usually satisfied for a stable numerical scheme. The third assumption assumes that the gradients of $\bu^i$ with respect to $\bt$ do not explode, which  makes the gradient-based optimizer break. The last assumption, on the other hand, assumes that the parameter $\bt$ is learnable: the maximum of $\left\| \frac{\partial\bu^i}{\partial\bt} \right\|$ cannot vanish otherwise there is insufficient update for $\bt$; meanwhile, the maximum $\rho(A_i^{-1})$ cannot be too small, otherwise the back-propagation will be damped too much and gradients cannot be efficiently transmitted. 
\end{remark}

The following two results are direct applications of \Cref{thm:main}
\begin{corollary}[Stability for the advection diffusion equation]
	The scheme proposed in \Cref{equ:scheme0} is stable.
\end{corollary}
\begin{proof}
	This is a special case of \Cref{coro:1} when $\alpha=1$.
\end{proof}

\begin{corollary}[Stability for the fractional Laplacian equation]
	The scheme proposed in \Cref{equ:fscheme} is stable.
\end{corollary}
\begin{proof}
	For the diffusion coefficient $a$, the analysis is essentially the same as above. Note that the space fractional index $s$ only resides in $A^s$ and $A^s$ does not depend on $\Delta t$, therefore, the assumptions are automatically satisfied as in the advection diffusion case.
\end{proof}

\begin{remark}
	We have developed the stability with emphasis on $\Delta t$. We could also develop stability concept with respect to $h$. However, since the stability due to the time evolution is most relevant for our problem, we focus on the stability with respect to $\Delta t$. 
\end{remark}

\begin{remark}
	An advantage of using implicit scheme is that the second assumption is satisfied in most schemes. We could also have certain stability conditions for explicit scheme. However, this could be troublesome as the optimizer evaluates the model for different parameters, which may lead to an inadmissible CFL condition.
\end{remark}

\section{Applications}\label{sect:ap}

As a demonstration of the flexibility of the proposed framework, we consider several models discussed in \Cref{sect:model problem}. Instead of implementing the  adjoint methods from scratch for each PDE constrained optimization model, we only need to substitute the hidden dynamics part, either implemented with custom operators or automatic differentiation. The decoupling is quite useful when we want to add more models. 

The settings for the hidden dynamics are the following. For the (fractional) advection diffusion equation, we adopt a uniform mesh for both time and space and let $\Delta\tau=0.01$, $h=3$. The degrees of freedom is $40\times 40$, with $40$ grid points per dimension. The time interval between two observations is $t_{i+1}-t_i=5\Delta \tau = 0.05$, i.e., we carry out the forward simulation for 5 times between two observations. The baseline quantity $\mathbf{m}_{\mathrm{base}}=3500$. After obtaining the $\mathbf{m}_{t_i}$ at various time, we upscale $\mathbf{m}_{t_i}$ to $150\times 150$ with bilinear interpolation to match the spatial configuration for the wave equation. We also assume $\mathbf{m}_0$ is given exactly. In addition, for testing the robustness, we add artificial i.i.d. zero-mean Gaussian noise to $\mathbf{m}_{t_i}$ before upscaling. The noise level is described by its standard deviation, which we denote $\sigma$ in the following. 

We now describe the settings for FWI. We discretize the domain $\Omega$ uniformly and the discretization points are $\bx_{ij} = [(i-1)h, (j-1)h]$, $i=1$, $2$, $\ldots$, $134$, $j=1$, $2$, $\ldots$, $384$ and $h=24$. The first dimension is the depth while the second is the distance. There are 3 phases of observations in total, where we denote as $\hat {\mathbf{d}_1}$, $\hat {\mathbf{d}_2}$ and $\hat {\mathbf{d}_3}$. In each phase, there are 30 sources and 295 receivers positioned at the depth of 768~m. The CPML boundaries with thickness of 32 grid points are placed along all four sides of the computational domain. In the inversion, we mask gradient updates in the CPML area. The source time function is a Ricker wavelet with a dominant frequency of 10~Hz. The simulation time is 5~s with a time interval of 0.0025~s.

For the optimizer, we use the built-in \texttt{L-BFGS-B} optimizer, which has been found very effective for many scientific computing problems. The stop criterion is whenever one of the following condition is satisfied: (1) The relative change in the objective function is less than $10^{-12}$; (2) The norm of projected gradient is less than $10^{-12}$; (3) maximum number of iteration 15000 is reached. The optimization for all cases converge within 30 iterations. Thanks to the \texttt{TensorFlow} backend parallelism feature, the inverse problem solver runs in parallel whenever two operations are separate in the computation dependency graph. For example, several FWI subroutines are automatically parallelized in \Cref{fig:backward}.

\subsection{Convergence Test for Automatic Differentiation}

Before we consider any applications, we perform the convergence test for automatic differentiation. Since we have considered a long and coupled dynamical system, if gradients are not implemented correctly, the error may accumulate during backpropagation. Therefore, it is crucial to verify the correctness of the gradients implementation. We apply the Taylor remainder convergence test~\cite{farrell2013automated}. Let $\bc$ be a scalar, vector or matrix, $F$ is a continuous functional that depends on $\bc$. The Taylor remainder test is based on the fact that given an arbitrary perturbation $\tilde {\bc}$ to $\bc$, we have
\begin{equation}\label{equ:first}
	\left| F(\bc+\gamma \tilde {\bc}) - F(\bc) \right| = \mathcal{O}(|\gamma|)
\end{equation}
while 
\begin{equation}\label{equ:second}
	\left| F(\bc+\gamma\tilde {\bc}) - F(\bc) - \gamma\langle\tilde {\bc}, \nabla F(\bc) \rangle \right| = \mathcal{O}(\gamma^2)
\end{equation}
where $\langle\cdot, \cdot \rangle$ denotes inner product. As long as $\langle\tilde {\bc}, \nabla F(\bc) \rangle\neq 0$, we should expect first order and second order convergence for \Cref{equ:first,equ:second}. \Cref{fig:grad_check} shows the convergence plot for the following operators respectively:
\begin{itemize}
	\item One step forward propagation of the advection diffusion equation, i.e., \Cref{equ:dis_ad}.
	\item One forward simulation of the wave equation with CPML, with the velocity field as inputs and the misfit as the output. 
\end{itemize}

In the plot, ``finite difference'' and ``automatic differentiation'' refer to \Cref{equ:first,equ:second} respectively. We see clearly that we obtain second order convergence for \Cref{equ:second}, indicating the correctness of the gradient computation.

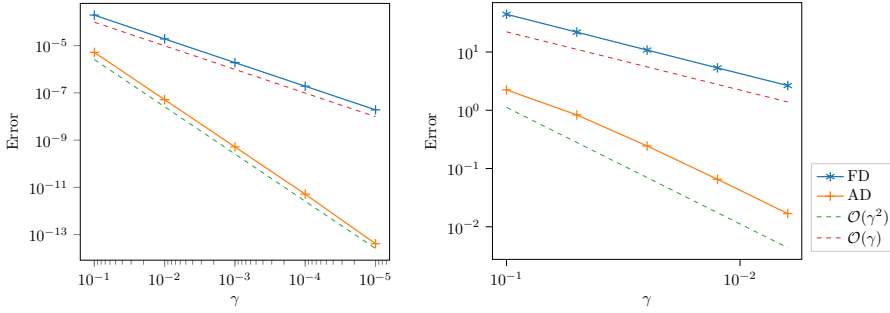
\begin{figure}[htpb]
\centering
\scalebox{0.6}{
\begin{tikzpicture}

\definecolor{color0}{rgb}{0.12156862745098,0.466666666666667,0.705882352941177}
\definecolor{color1}{rgb}{1,0.498039215686275,0.0549019607843137}
\definecolor{color2}{rgb}{0.172549019607843,0.627450980392157,0.172549019607843}
\definecolor{color3}{rgb}{0.83921568627451,0.152941176470588,0.156862745098039}

\begin{axis}[
legend cell align={left},
tick align=outside,
x dir=reverse,
tick pos=left,
x grid style={white!69.01960784313725!black},
xlabel={$\gamma$},
xmin=6.30957344480193e-06, xmax=0.158489319246111,
xmode=log,
xtick={1e-07,1e-06,1e-05,0.0001,0.001,0.01,0.1,1,10},
y grid style={white!69.01960784313725!black},
ylabel={Error},
ymin=8.28958938573402e-15, ymax=0.000613005973980258,
ymode=log,
ytick={1e-17,1e-15,1e-13,1e-11,1e-09,1e-07,1e-05,0.001,0.1},
]
\addlegendimage{no markers, color0}
\addlegendimage{no markers, color1}
\addlegendimage{no markers, color2}
\addlegendimage{no markers, color3}
\addplot [thick, color0, mark=+, mark size=3, mark options={solid}]
table [row sep=\\]{%
0.1	0.000196527060088769 \\
0.01	1.9187284381772e-05 \\
0.001	1.9140740707968e-06 \\
0.0001	1.91360852497269e-07 \\
1e-05	1.91356104295437e-08 \\
};
\addplot [thick, color1, mark=+, mark size=3, mark options={solid}]
table [row sep=\\]{%
0.1	5.17136705042347e-06 \\
0.01	5.17150779374014e-08 \\
0.001	5.17140413343352e-10 \\
0.0001	5.15945892303326e-12 \\
1e-05	4.11257091289924e-14 \\
};
\addplot [semithick, color2, dashed]
table [row sep=\\]{%
0.1	2.58568352521174e-06 \\
0.01	2.58568352521173e-08 \\
0.001	2.58568352521173e-10 \\
0.0001	2.58568352521173e-12 \\
1e-05	2.58568352521173e-14 \\
};
\addplot [semithick, color3, dashed]
table [row sep=\\]{%
0.1	9.82635300443846e-05 \\
0.01	9.82635300443846e-06 \\
0.001	9.82635300443846e-07 \\
0.0001	9.82635300443846e-08 \\
1e-05	9.82635300443846e-09 \\
};
\end{axis}

\end{tikzpicture}}~
\scalebox{0.6}{
\begin{tikzpicture}

\definecolor{color0}{rgb}{0.12156862745098,0.466666666666667,0.705882352941177}
\definecolor{color1}{rgb}{1,0.498039215686275,0.0549019607843137}
\definecolor{color2}{rgb}{0.172549019607843,0.627450980392157,0.172549019607843}
\definecolor{color3}{rgb}{0.83921568627451,0.152941176470588,0.156862745098039}

\begin{axis}[
legend cell align={left},
legend style={at={(1.03,0.03)}, anchor=south west, draw=white!80.0!black},
log basis x={10},
log basis y={10},
x dir=reverse,
tick align=outside,
tick pos=left,
x grid style={white!69.01960784313725!black},
xlabel={\(\displaystyle \gamma\)},
xmin=0.00544094102060078, xmax=0.114869835499704,
xmode=log,
xtick style={color=black},
xtick={0.0001,0.001,0.01,0.1,1,10},
xticklabels={\(\displaystyle {10^{-4}}\),\(\displaystyle {10^{-3}}\),\(\displaystyle {10^{-2}}\),\(\displaystyle {10^{-1}}\),\(\displaystyle {10^{0}}\),\(\displaystyle {10^{1}}\)},
y grid style={white!69.01960784313725!black},
ylabel={Error},
ymin=0.00275657277860811, ymax=70.4787965880313,
ymode=log,
ytick style={color=black},
ytick={0.0001,0.001,0.01,0.1,1,10,100,1000},
yticklabels={\(\displaystyle {10^{-4}}\),\(\displaystyle {10^{-3}}\),\(\displaystyle {10^{-2}}\),\(\displaystyle {10^{-1}}\),\(\displaystyle {10^{0}}\),\(\displaystyle {10^{1}}\),\(\displaystyle {10^{2}}\),\(\displaystyle {10^{3}}\)}
]
\addplot [thick, color0, mark=asterisk, mark size=3, mark options={solid}]
table {%
0.1 44.4333419799805
0.05 21.9299163818359
0.025 10.7926025390625
0.0125 5.33976745605469
0.00625 2.65400695800781
};
\addlegendentry{FD}
\addplot [thick, color1, mark=+, mark size=3, mark options={solid}]
table {%
0.1 2.23866404877565
0.05 0.832577416233526
0.025 0.243933056261294
0.0125 0.0654327146540847
0.00625 0.0168395873075111
};
\addlegendentry{AD}
\addplot [semithick, color2, dashed]
table {%
0.1 1.11933202438782
0.05 0.279833006096956
0.025 0.069958251524239
0.0125 0.0174895628810597
0.00625 0.00437239072026493
};
\addlegendentry{$\mathcal{O}(\gamma^2)$}
\addplot [semithick, color3, dashed]
table {%
0.1 22.2166709899902
0.05 11.1083354949951
0.025 5.55416774749756
0.0125 2.77708387374878
0.00625 1.38854193687439
};
\addlegendentry{$\mathcal{O}(\gamma)$}
\end{axis}

\end{tikzpicture}}~
\caption{Convergence plot for the gradients computed via finite difference~(FD) and automatic differentiation~(AD). ``Finite difference'' and ``automatic differentiation'' refers to \Cref{equ:first,equ:second} respectively. The first plot corresponds to one step forward propagation of the advection diffusion equation; the second corresponds to one forward simulation of the wave equation with CPML.}
\label{fig:grad_check}
\end{figure}

\subsection{Advection Diffusion Equation}

In \Cref{fig:us}, we show the basic settings. The first plot shows the initial configuration for $m$, which consists of 5 piecewise constant domains, surrounded by zero margins. The following plots to the right shows the result at $t=0.25$ with different noise levels $\sigma=0.0$, $5.0$, $10.0$. The corresponding governing equation is 
$$\frac{\partial m}{\partial t} = 10\Delta m + 0.1 \frac{\partial m}{\partial x} - 0.2 \frac{\partial m}{\partial y}$$
with zero Dirichlet boundary condition. The i.i.d.\ Gaussian noise is added to $m$
\begin{equation}
	\tilde{\mathbf{m}}_{t_i} = \mathbf{m}_{t_i} +  \mathbf{W}_i, \quad  \mathbf{W}_i\sim \mathcal{N}(\mathbf{0}, \sigma^2\mathbf{I})
\end{equation}

\begin{figure}[htpb]
  \includegraphics[width=1.0\textwidth]{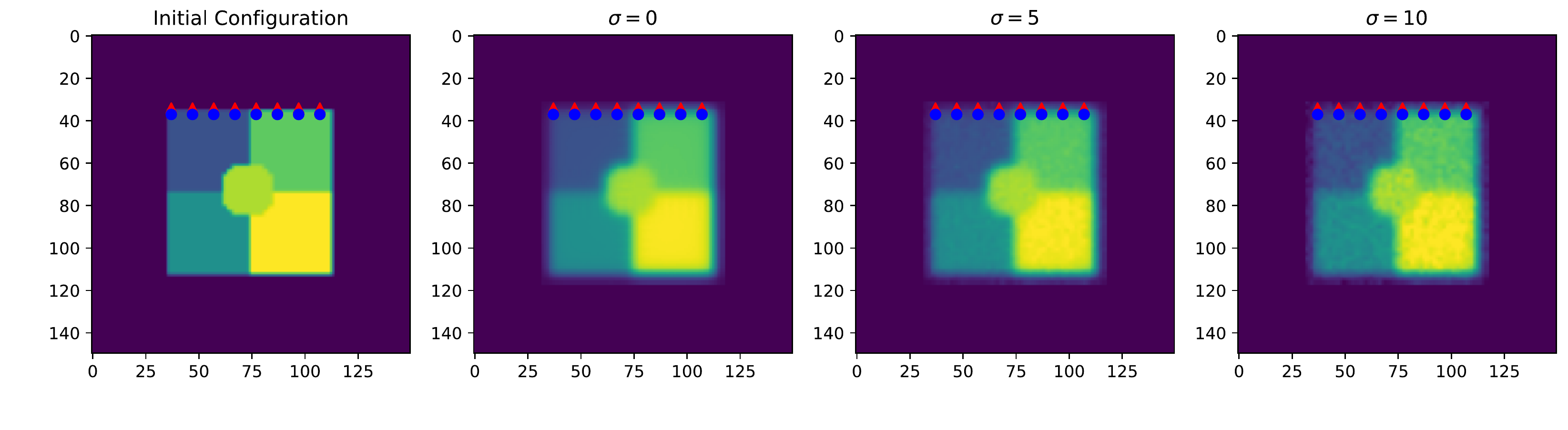}
  \caption{Basic settings for the advection diffusion equation. The red triangles (\textcolor{red}{$\blacktriangle$}) represent sources while the blue dots (\textcolor{blue}{$\bullet$}) represent receivers. The first figure shows the initial configuration $\mathbf{m}_0$, the others show the terminal configuration for different noise levels. We can see that the edges (sudden changes in $m(\cdot,t)$) are blurrier when the noise level $\sigma$ is larger.}
  \label{fig:us}
\end{figure}

In \Cref{tab:ad} we show the results for the advection diffusion equation. We can see that for the no-noise case, we recover the coefficients exactly. As we increase the noise, the converged values start to deviate from the true ones but still remain quite robust against noise. 

\begin{table}[htpb]
\centering
\begin{tabular}{@{}llll@{}}
\toprule
Governing Equation & $\sigma=0$ & $\sigma=5$ & $\sigma=10$ \\ \midrule
$\frac{\partial m}{\partial t} = \Delta m + 0.1 \frac{\partial m}{\partial x} - 0.2 \frac{\partial m}{\partial y}$ & \makecell{$a/a^*\ =1.0000$ \\ $b_1/{b_1^*}=1.0000$ \\ $b_2/{b_2^*}=1.0000$} & \makecell{$a/a^*\ =0.8740$ \\ $b_1/{b_1^*}=0.9782$ \\ $b_2/{b_2^*}=1.0146$} & \makecell{$a/a^*\ =0.7201$ \\ $b_1/{b_1^*}=0.9535$ \\ $b_2/{b_2^*}=0.9560$} \\ \hline
$\frac{\partial m}{\partial t} = 10\Delta m + 0.1 \frac{\partial m}{\partial x} - 0.2 \frac{\partial m}{\partial y}$ & \makecell{$a/a^*\ =1.0000$ \\ $b_1/{b_1^*}=1.0000$ \\ $b_2/{b_2^*}=1.0000$} & \makecell{$a/a^*\ =0.9773$ \\ $b_1/{b_1^*}=1.0022$ \\ $b_2/{b_2^*}=1.0409$} & \makecell{$a/a^*\ =0.9472$ \\ $b_1/{b_1^*}=0.9497$ \\ $b_2/{b_2^*}=0.9261$} \\ \hline
$\frac{\partial m}{\partial t} = 100\Delta m + 0.1 \frac{\partial m}{\partial x} - 0.2 \frac{\partial m}{\partial y}$&  \makecell{$a/a^*\ =1.0000$ \\ $b_1/{b_1^*}=1.0000$ \\ $b_2/{b_2^*}=1.0000$} & \makecell{$a/a^*\ =0.9808$ \\ $b_1/{b_1^*}=0.8611$ \\ $b_2/{b_2^*}=0.9845$} &  \makecell{$a/a^*\ =1.0357$ \\ $b_1/{b_1^*}=1.0560$ \\ $b_2/{b_2^*}=0.9172$}\\ \bottomrule
\end{tabular}
\caption{Result for advection diffusion equation. The first column shows the governing equations we use to generate synthetic observation data. $a^*$, $b_1^*$ and $b_2^*$ are exact values. $\sigma$ is the standard deviation in the Gaussian noise added to $m$.}
\label{tab:ad}
\end{table}

\subsection{Time fractional Advection Diffusion Equation}

For the time fractional advection diffusion case, we need to point out that the discrete adjoint state equation is much more convoluted than its integer order counterparts to derive and implement. We refer readers to \cite{antil2015fractional,maryshev2013adjoint} for adjoint methods on fractional partial differential equations. Machine learning techniques have also been applied to calibrate the space fractional index~\cite{gulian2018machine}.

\Cref{tab:ft-ad} shows the result for the time-fractional advection diffusion equation. We can see that for most cases the method discovers the true value quite accurately. However, for large noise and small $\alpha^*$, we observed that the calibration of $a$ has more than 20\% error. If $\alpha$ is too small, the dynamical system exhibits super-diffusion, and we need a larger computational domain to capture the changes. Noise has negative impact on the accuracy of the calibration. As long as $\alpha$ is reasonable large and  noise is not too large, the proposed method shows good performance on estimating the parameters, especially for the fractional indices.

\begin{table}[htpb]
\centering
\begin{tabular}{@{}llll@{}}
\toprule
Governing Equation & $\sigma=0$ & $\sigma=5$ & $\sigma=10$ \\ \midrule
${}_0^CD_t^{0.8}m = 10\Delta m $ & \makecell{$a/a^*\ =1.0000$ \\  $\quad\alpha\quad =\mathbf{0.8000}$} & \makecell{$a/a^*\ =0.9109$ \\  $\quad\alpha\quad =\mathbf{0.7993}$} & \makecell{$a/a^*\ =1.0973$ \\  $\quad\alpha\quad =\mathbf{0.8030}$}  \\ \hline
${}_0^CD_t^{0.6}m = 10\Delta m $ & \makecell{$a/a^*\ =1.0000$ \\  $\quad\alpha\quad =\mathbf{0.6000}$} & \makecell{$a/a^*\ =0.8388$ \\  $\quad\alpha\quad =\mathbf{0.5927}$} & \makecell{$a/a^*\ =1.1312$ \\  $\quad\alpha\quad =\mathbf{0.6061}$}\\ \hline
${}_0^CD_t^{0.4}m = 10\Delta m $ & \makecell{$a/a^*\ =1.0000$ \\  $\quad\alpha\quad =\mathbf{0.4000}$} & \makecell{$a/a^*\ =1.2763$ \\  $\quad\alpha\quad =\mathbf{0.4081}$} & \makecell{$a/a^*\ ={0.8447}$ \\  $\quad\alpha\quad = \mathbf{0.3908}$} \\ \hline
${}_0^CD_t^{0.2}m = 10\Delta m $ & \makecell{$a/a^*\ =0.9994$ \\  $\quad\alpha\quad =\mathbf{0.2000}$} & \makecell{$a/a^*\ =0.3474$ \\  $\quad\alpha\quad =\mathbf{0.1826}$} & \makecell{$a/a^*\ ={0.8633}$ \\  $\quad\alpha\quad =\mathbf{0.2064}$} \\   \bottomrule
\end{tabular}
\caption{Result for time-fractional advection diffusion equation. The first column shows the governing equations we use to generate synthetic observation data. $a^*$ is the exact value. $\sigma$ is the standard deviation in the Gaussian noise added to $m$, $\alpha$ is the estimated time fractional index. For large noise and small $\alpha^*$ and smaller $\alpha$, in this case the sub-diffusion effects dominate, the presence of noise makes it difficult to estimate the diffusion coefficient. It is remarkable even in the presence of large noise $\sigma=10$, the estimation of $\alpha$ is still accurate.}
\label{tab:ft-ad}
\end{table}

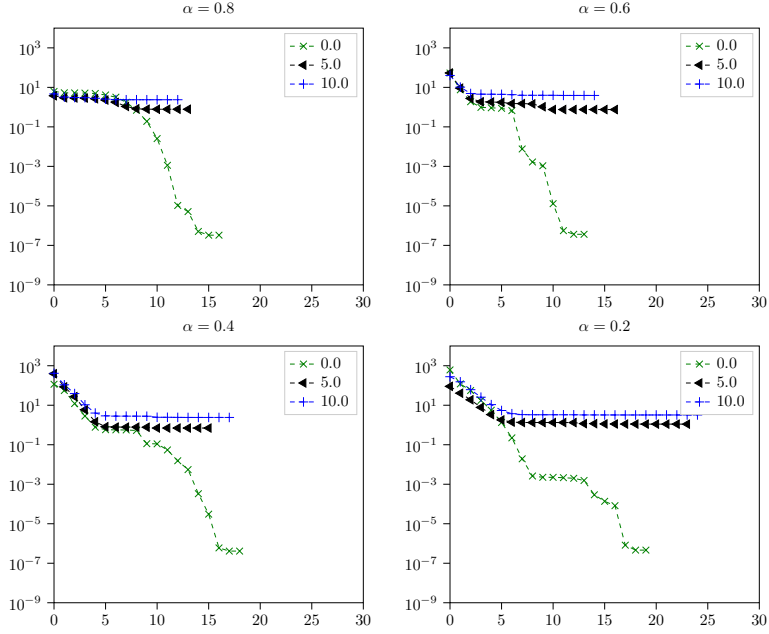
\begin{figure}[htbp]
\centering
\scalebox{0.6}{
\begin{tikzpicture}

\begin{axis}[
legend cell align={left},
legend style={draw=white!80.0!black},
log basis y={10},
tick align=outside,
tick pos=left,
title={$\alpha=0.8$},
x grid style={white!69.01960784313725!black},
xmin=0, xmax=30,
xtick style={color=black},
y grid style={white!69.01960784313725!black},
ymin=1e-09, ymax=10000,
ymode=log,
ytick style={color=black},
ytick={1e-11,1e-09,1e-07,1e-05,0.001,0.1,10,1000,100000,10000000},
yticklabels={\(\displaystyle {10^{-11}}\),\(\displaystyle {10^{-9}}\),\(\displaystyle {10^{-7}}\),\(\displaystyle {10^{-5}}\),\(\displaystyle {10^{-3}}\),\(\displaystyle {10^{-1}}\),\(\displaystyle {10^{1}}\),\(\displaystyle {10^{3}}\),\(\displaystyle {10^{5}}\),\(\displaystyle {10^{7}}\)}
]
\addplot [semithick, green!50.0!black, dashed, mark=x, mark size=3, mark options={solid}]
table {%
0 6.20019686222076
1 5.31818246841431
2 5.23593926429749
3 5.1429318189621
4 4.90514659881592
5 4.11236894130707
6 3.1892848610878
7 1.75692939758301
8 0.691647581756115
9 0.190195439383388
10 0.0255415043793619
11 0.00113238350604661
12 1.04633927549003e-05
13 5.13777843025309e-06
14 5.06995490923146e-07
15 3.28452085085473e-07
16 3.28452085085473e-07
};
\addlegendentry{0.0}
\addplot [semithick, black, dashed, mark=triangle*, mark size=3, mark options={solid,rotate=90}]
table {%
0 3.74137800931931
1 2.96497243642807
2 2.90536007285118
3 2.86921134591103
4 2.64516565203667
5 2.28692671656609
6 1.73966497182846
7 1.0819331407547
8 0.823350414633751
9 0.786543846130371
10 0.781420558691025
11 0.781021609902382
12 0.781016424298286
13 0.781016424298286
};
\addlegendentry{5.0}
\addplot [semithick, blue, dashed, mark=+, mark size=3, mark options={solid}]
table {%
0 4.66168576478958
1 3.21970421075821
2 3.06115490198135
3 3.00740933418274
4 2.80705672502518
5 2.54541540145874
6 2.38585436344147
7 2.36647206544876
8 2.36614871025085
9 2.3661293387413
10 2.36611729860306
11 2.36611723899841
12 2.36611723899841
};
\addlegendentry{10.0}
\end{axis}

\end{tikzpicture}}~
\scalebox{0.6}{
\begin{tikzpicture}

\begin{axis}[
legend cell align={left},
legend style={draw=white!80.0!black},
log basis y={10},
tick align=outside,
tick pos=left,
title={$\alpha=0.6$},
x grid style={white!69.01960784313725!black},
xmin=0, xmax=30,
xtick style={color=black},
y grid style={white!69.01960784313725!black},
ymin=1e-09, ymax=10000,
ymode=log,
ytick style={color=black},
ytick={1e-11,1e-09,1e-07,1e-05,0.001,0.1,10,1000,100000,10000000},
yticklabels={\(\displaystyle {10^{-11}}\),\(\displaystyle {10^{-9}}\),\(\displaystyle {10^{-7}}\),\(\displaystyle {10^{-5}}\),\(\displaystyle {10^{-3}}\),\(\displaystyle {10^{-1}}\),\(\displaystyle {10^{1}}\),\(\displaystyle {10^{3}}\),\(\displaystyle {10^{5}}\),\(\displaystyle {10^{7}}\)}
]
\addplot [semithick, green!50.0!black, dashed, mark=x, mark size=3, mark options={solid}]
table {%
0 53.3328533172607
1 7.52303129434586
2 1.87733617424965
3 0.969713598489761
4 0.924802258610725
5 0.871628478169441
6 0.665938287973404
7 0.00759840616956353
8 0.00166446508956142
9 0.00107321434188634
10 1.29083266529051e-05
11 5.55726316520122e-07
12 3.63773594358463e-07
13 3.63773594358463e-07
};
\addlegendentry{0.0}
\addplot [semithick, black, dashed, mark=triangle*, mark size=3, mark options={solid,rotate=90}]
table {%
0 53.5708150863647
1 9.12848258018494
2 2.73343834280968
3 1.86155089735985
4 1.81265047192574
5 1.75044342875481
6 1.50706508755684
7 1.50678873062134
8 1.45634807646275
9 1.04243601858616
10 0.753917559981346
11 0.75364588201046
12 0.752063199877739
13 0.752062648534775
14 0.752056166529655
15 0.752055391669273
16 0.752055391669273
};
\addlegendentry{5.0}
\addplot [semithick, blue, dashed, mark=+, mark size=3, mark options={solid}]
table {%
0 40.1463875770569
1 10.2820897102356
2 4.90536141395569
3 4.52179956436157
4 4.49557548761368
5 4.44040274620056
6 4.27240079641342
7 3.97211766242981
8 3.97207605838776
9 3.97125858068466
10 3.96132272481918
11 3.9042734503746
12 3.90096569061279
13 3.89880645275116
14 3.89880645275116
};
\addlegendentry{10.0}
\end{axis}

\end{tikzpicture}}
\scalebox{0.6}{
\begin{tikzpicture}

\begin{axis}[
legend cell align={left},
legend style={draw=white!80.0!black},
log basis y={10},
tick align=outside,
tick pos=left,
title={$\alpha=0.4$},
x grid style={white!69.01960784313725!black},
xmin=0, xmax=30,
xtick style={color=black},
y grid style={white!69.01960784313725!black},
ymin=1e-09, ymax=10000,
ymode=log,
ytick style={color=black},
ytick={1e-11,1e-09,1e-07,1e-05,0.001,0.1,10,1000,100000,10000000},
yticklabels={\(\displaystyle {10^{-11}}\),\(\displaystyle {10^{-9}}\),\(\displaystyle {10^{-7}}\),\(\displaystyle {10^{-5}}\),\(\displaystyle {10^{-3}}\),\(\displaystyle {10^{-1}}\),\(\displaystyle {10^{1}}\),\(\displaystyle {10^{3}}\),\(\displaystyle {10^{5}}\),\(\displaystyle {10^{7}}\)}
]
\addplot [semithick, green!50.0!black, dashed, mark=x, mark size=3, mark options={solid}]
table {%
0 115.451650619507
1 55.6815557479858
2 12.1681172847748
3 2.84161776304245
4 0.774292968213558
5 0.57598253339529
6 0.565692208707333
7 0.559265069663525
8 0.527270138263702
9 0.113389445468783
10 0.110731847584248
11 0.0544058671221137
12 0.0154661547858268
13 0.00551108899526298
14 0.000330900664266665
15 3.06219626509119e-05
16 5.98192627876415e-07
17 4.12627102264196e-07
18 4.12627102264196e-07
};
\addlegendentry{0.0}
\addplot [semithick, black, dashed, mark=triangle*, mark size=3, mark options={solid,rotate=90}]
table {%
0 393.440849304199
1 84.7782745361328
2 27.2842655181885
3 5.89879727363586
4 1.44799050688744
5 0.817721888422966
6 0.785089984536171
7 0.783622309565544
8 0.780330285429955
9 0.770978480577469
10 0.701835945248604
11 0.700526863336563
12 0.700445726513863
13 0.700433939695358
14 0.700425520539284
15 0.700425520539284
};
\addlegendentry{5.0}
\addplot [semithick, blue, dashed, mark=+, mark size=3, mark options={solid}]
table {%
0 414.09595489502
1 108.262996673584
2 39.0302383899689
3 10.5616307258606
4 4.02375507354736
5 2.8948495388031
6 2.81223332881927
7 2.80774056911469
8 2.80121368169785
9 2.7783197760582
10 2.45529371500015
11 2.45418035984039
12 2.42654532194138
13 2.41067522764206
14 2.40962153673172
15 2.40808022022247
16 2.40806806087494
17 2.408067882061
};
\addlegendentry{10.0}
\end{axis}

\end{tikzpicture}}~
\scalebox{0.6}{
\begin{tikzpicture}

\begin{axis}[
legend cell align={left},
legend style={draw=white!80.0!black},
log basis y={10},
tick align=outside,
tick pos=left,
title={$\alpha=0.2$},
x grid style={white!69.01960784313725!black},
xmin=0, xmax=30,
xtick style={color=black},
y grid style={white!69.01960784313725!black},
ymin=1e-09, ymax=10000,
ymode=log,
ytick style={color=black},
ytick={1e-11,1e-09,1e-07,1e-05,0.001,0.1,10,1000,100000,10000000},
yticklabels={\(\displaystyle {10^{-11}}\),\(\displaystyle {10^{-9}}\),\(\displaystyle {10^{-7}}\),\(\displaystyle {10^{-5}}\),\(\displaystyle {10^{-3}}\),\(\displaystyle {10^{-1}}\),\(\displaystyle {10^{1}}\),\(\displaystyle {10^{3}}\),\(\displaystyle {10^{5}}\),\(\displaystyle {10^{7}}\)}
]
\addplot [semithick, green!50.0!black, dashed, mark=x, mark size=3, mark options={solid}]
table {%
0 613.23722076416
1 120.114232063293
2 55.5661678314209
3 17.4515013694763
4 5.57527741789818
5 1.34205881506205
6 0.222799475304782
7 0.0193450364749879
8 0.00261068153486121
9 0.00224743891158141
10 0.00223327704588883
11 0.00214667996624485
12 0.00200046303507406
13 0.00155005925626028
14 0.000291554782961612
15 0.000136574173211557
16 8.2236731032026e-05
17 8.23142755734807e-07
18 4.62587259164593e-07
19 4.62587259164593e-07
};
\addlegendentry{0.0}
\addplot [semithick, black, dashed, mark=triangle*, mark size=3, mark options={solid,rotate=90}]
table {%
0 90.5301990509033
1 40.8014607429504
2 18.8967761993408
3 7.7772341966629
4 3.35537177324295
5 1.8030880689621
6 1.40043035149574
7 1.33696220815182
8 1.33271856606007
9 1.33240655064583
10 1.33158564567566
11 1.32926990091801
12 1.32199600338936
13 1.1702833622694
14 1.16858810186386
15 1.13299213349819
16 1.12438536435366
17 1.11624400317669
18 1.11591820418835
19 1.10879699885845
20 1.10864049196243
21 1.10855661332607
22 1.10854998230934
23 1.10854998230934
};
\addlegendentry{5.0}
\addplot [semithick, blue, dashed, mark=+, mark size=3, mark options={solid}]
table {%
0 278.453353881836
1 153.595050811768
2 61.4162473678589
3 25.5393863916397
4 10.7575808167458
5 5.41261875629425
6 3.73472630977631
7 3.34917116165161
8 3.29887133836746
9 3.29619336128235
10 3.29560053348541
11 3.29301160573959
12 3.28753006458282
13 3.23356175422668
14 3.23348748683929
15 3.22321617603302
16 3.22196024656296
17 3.21674746274948
18 3.21599471569061
19 3.21566373109818
20 3.21562641859055
21 3.21560418605804
22 3.21558380126953
23 3.21558153629303
24 3.21557992696762
};
\addlegendentry{10.0}
\end{axis}

\end{tikzpicture}}
\caption{Loss functions for different $\alpha$'s and different noise levels. The legend numbers denote noise levels. \revise{Due to noise, the loss will converge to different levels and we obtain larger terminal losses for larger noise levels $\sigma$.}}
\label{fig:g}
\end{figure}

We also show the loss function for the cases shown above in \Cref{fig:g}. We can see that for given $\alpha$, for larger noise levels, the optimizer terminates at larger losses. For a fixed noise level, we obtain faster convergence for larger $\alpha$'s. These are consistent with our discussion above.

\subsection{Space-fractional Advection-Diffusion Equation}

\begin{table}[htpb]
\centering
\begin{tabular}{@{}llll@{}}
\toprule
Governing Equation & $\sigma=0$ & $\sigma=5$ & $\sigma=10$ \\ \midrule
$\frac{\partial m}{\partial t} = -10(-\Delta)^{0.2} m$ & \makecell{$a/a^*\ =1.0000$ \\   $\quad s\quad =\mathbf{0.2000}$} & \makecell{$a/a^*\ =1.0378$ \\   $\quad s\quad =\mathbf{0.2069}$} & \makecell{$a/a^*\ =1.0948$ \\   $\quad s\quad =\mathbf{0.2159}$} \\  \hline
$\frac{\partial m}{\partial t} = -10(-\Delta)^{0.4} m$ & \makecell{$a/a^*\ =1.0000$ \\   $\quad s\quad =\mathbf{0.4000}$} & \makecell{$a/a^*\ =0.9834$ \\   $\quad s\quad =\mathbf{0.3983}$} & \makecell{$a/a^*\ =0.9900$ \\   $\quad s\quad =\mathbf{0.3946}$}\\  \hline
$\frac{\partial m}{\partial t} = -10(-\Delta)^{0.6} m$ & \makecell{$a/a^*\ =1.0000$ \\   $\quad s\quad =\mathbf{0.6000}$} & \makecell{$a/a^*\ =1.0285$ \\   $\quad s\quad =\mathbf{0.6021}$} & \makecell{$a/a^*\ =0.9657$ \\   $\quad s\quad =\mathbf{0.5807}$}\\  \hline
$\frac{\partial m}{\partial t} = -10(-\Delta)^{0.8} m$ & \makecell{$a/a^*\ =1.0000$ \\   $\quad s\quad =\mathbf{0.8000}$} & \makecell{$a/a^*\ =1.0365$ \\   $\quad s\quad =\mathbf{0.8093}$} & \makecell{$a/a^*\ =0.9649$ \\   $\quad s\quad =\mathbf{0.7675}$}\\  \bottomrule
\end{tabular}
\caption{Result for space-fractional advection diffusion equation. The first column shows the governing equations we use to generate synthetic observation data. $a^*$ is the exact value. $\sigma$ is the standard deviation in the Gaussian noise added to $m$, $s$ is the estimated space fractional index.}
\label{tab:st-ad}
\end{table}

Finally we show the result for the space-fractional advection diffusion equation in \Cref{tab:st-ad}. We can see that for $\sigma=0$, $5$, $10$, the method recovers both the fractional indices and the diffusion coefficients with good accuracy.

\section{Conclusion}\label{sect:conc}

We developed a new algorithm to learn the hidden dynamics from indirect observations by solving a PDE constrained optimization problem. The algorithm uses a gradient-based optimization method. This is achieved through a flexible gradient back-propagation scheme. The computation tool, \texttt{ADCME}, which materializes this framework, provides the user with a friendly syntax for expressing scientific computing algorithms. It has a powerful and flexible interface. It was designed from the ground up such that custom built differentiable operators can be integrated in a modular fashion. It can be found at
\begin{center}
    \url{https://github.com/kailaix/ADCME.jl}
\end{center}
The \texttt{TensorFlow} backend provides automatic differentiation as an alternative way to deriving and implementing adjoint methods.

The numerical results showed the effectiveness of the framework and the software. We coupled different hidden dynamics (fractional and advection-diffusion) with an acoustic wave equation. We found that the algorithms could learn the hidden dynamics accurately in most cases on the test problems. In some cases, such as when we had very small time-fractional indices and large noise, the diffusion coefficients were estimated less accurately. This was to be expected because of the ill-conditioning of the mathematical problem.

We showed proofs of effectiveness on a time-lapse monitoring problem in seismic imaging. However, the framework is not limited to the full-waveform inversion of slow-time-scale processes. This framework can solve many other PDE constrained optimization problems by decoupling the forward simulation into individual operations. In the future, we plan to benchmark it on geophysics problems where the hidden dynamics are usually assumed to be more specific and complicated models such as the two-phase fluid flow models or the black oil model.
 

\section*{Acknowledgements}
We thank Prof.~Hamdi Tchelepi for inspiring discussion on automatic differentiation and Stephan Hoyer (Google) for help with \texttt{TensorFlow}. We thank Prof.~Tapan Mukerji for helpful discussions on time-lapse seismic inversion and reservoir characterizations. Kailai Xu also thanks the Stanford Graduate Fellowship in Science \& Engineering and the 2018 Schlumberger Innovation Fellowship for their financial support.

\newpage

\appendix
\section{Sample Codes for Wave Equations}\label{sect:code}

We show in \Cref{fig:code} how to solve wave equation with \texttt{ADCME.jl}. The implementation is based on \cite{grote2010efficient}.
\begin{figure}[hbtp]
\centering
  \includegraphics[width=1.0\textwidth]{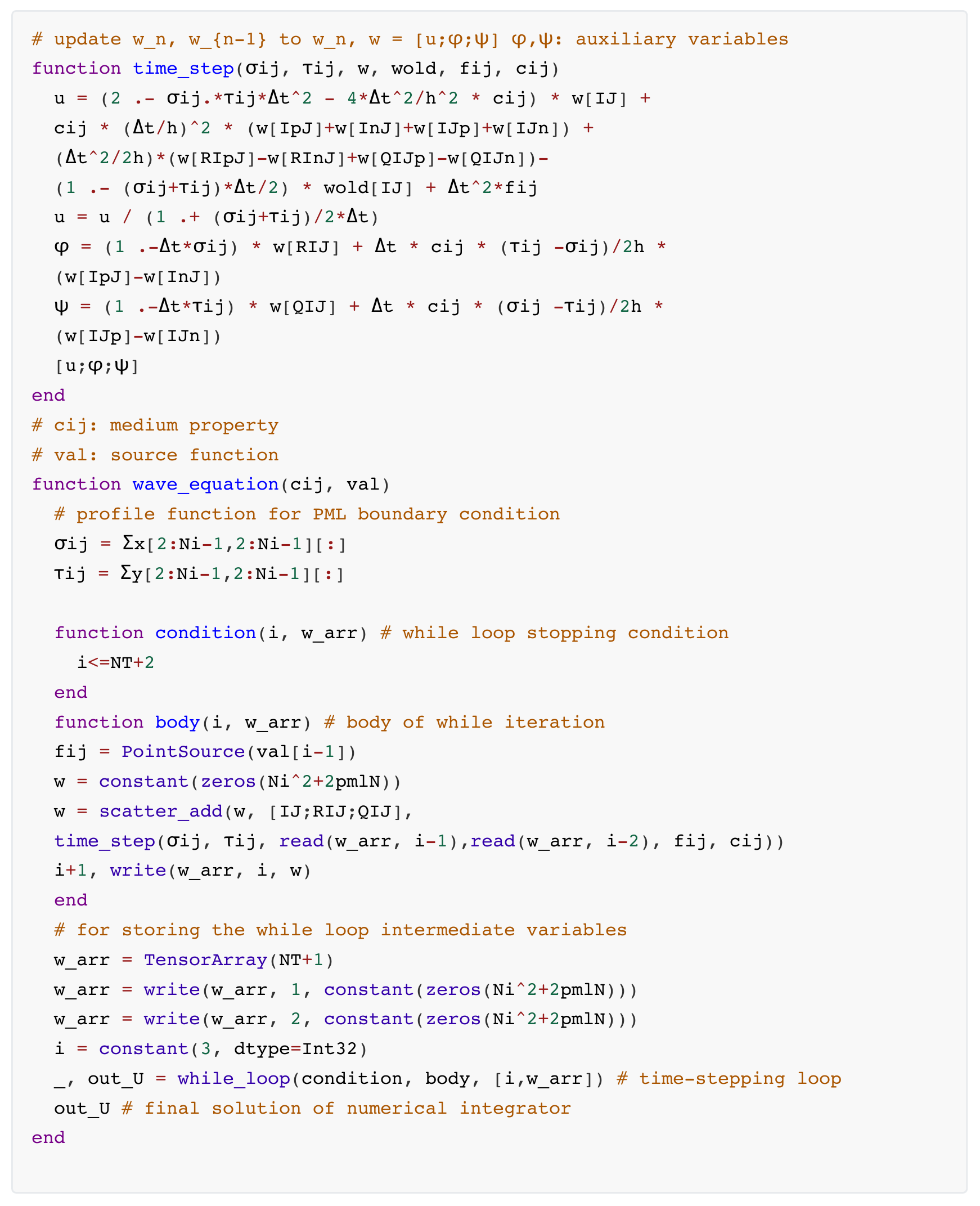}
  \caption{Sample codes for wave equations. The contexts and some global variables are omitted for conciseness.}
  \label{fig:code}
\end{figure}

\section{Proof of \Cref{thm:main}}\label{sect:proof}

\begin{proof}
For clarity we introduce several notation
\begin{align}
	{Q_i} :=& \left\| {\frac{{D{{\bu}^i}}}{{D\bt}}} \right\|\\
	{\rho _{\max }} :=& \mathop {\max }\limits_{0 \leqslant i \leqslant n} \rho (A_i^{ - 1})\\
	\Theta  :=& \mathop {\max }\limits_{0 \leqslant i \leqslant n} \left\| {\frac{{\partial {{\bu}^i}}}{{\partial \bt}}} \right\|
\end{align}
We have
\begin{equation}\label{equ:proof1}
  \frac{{D{{\bu}^i}}}{{D\bt}} = \frac{{\partial {{\bu}^i}}}{{\partial \bt}} + \frac{{\partial {{\bu}^i}}}{{\partial {{\bu}^{i - 1}}}}\frac{{D{{\bu}^{i - 1}}}}{{D\bt}} + \frac{{\partial {{\bu}^i}}}{{\partial {{\bu}^{i - 2}}}}\frac{{D{{\bu}^{i - 2}}}}{{D\bt}} +  \cdots  + \frac{{\partial {{\bu}^i}}}{{\partial {{\bu}^0}}}\frac{{D{{\bu}^0}}}{{D\bt}}
\end{equation}
By \Cref{equ:Fiex} we have
$${{\bu}^i} = A_i^{ - 1}\left( {\sum\limits_{j = 0}^{i - 1} {{a_{ij}}{{\bu}^j}} } \right)$$
and therefore
$$\frac{{\partial {{\bu}^i}}}{{\partial {{\bu}^j}}} = {a_{ij}}A_i^{ - 1}$$
Plug it into \Cref{equ:proof1} and notice that $\rho(A_i^{-1})\leq \rho_{\max}$ we have
$${Q_i} = \left\| {\frac{{D{{\bu}^i}}}{{D\bt}}} \right\| \leqslant \Theta  + {\rho _{\max }}\sum\limits_{j = 0}^{i - 1} {{a_{ij}}{Q_j}} $$
Let $\tilde Q_i$ be auxiliary variables which satisfies
\begin{align}
	{{\tilde Q}_0} =& {Q_0}\\
	{{\tilde Q}_i} =& \Theta  + {\rho _{\max }}\sum\limits_{j = 0}^{i - 1} {{a_{ij}}{{\tilde Q}_j}} 
\end{align}
We now prove that $\tilde Q_j\geq Q_j$, $\forall j$ by mathematical induction.
\begin{itemize}
	\item The claim is true for $j=0$.
	\item Assume it is true for all $j\leq i-1$, for $j=i$, we have
	$${Q_i} \leqslant \Theta  + {\rho _{\max }}\sum\limits_{j = 0}^{i - 1} {{a_{ij}}{Q_j}}  \leqslant \Theta  + {\rho _{\max }}\sum\limits_{j = 0}^{i - 1} {{a_{ij}}{{\tilde Q}_j}}  \leqslant {{\tilde Q}_i}$$
	hence it is also true for $j=i$.
	\item As a result, $\tilde Q_j\geq Q_j$, $\forall j$.
\end{itemize}

Now we give a bound on $\tilde Q_j$. We scale $\tilde Q_j$ by a factor of $\frac{\Theta}{1-\rho_{\max}}$, i.e.,
$$Q{'_i} = \frac{{{{\tilde Q}_i}\left( {1 - {\rho _{\max }}} \right)}}{\Theta }$$
then we have
$$Q{'_i} = 1 - {\rho _{\max }} + {\rho _{\max }}\sum\limits_{j = 0}^{i - 1} {{a_{ij}}Q{'_j}} $$
We prove that if $Q'_0\leq C$, where $C>0$ is a constant, then $Q'_i\leq C$, $\forall i$. In fact, thanks to $C>1$ and the first assumption in the theorem, if $Q'_j<C$ for $j<i$, then we have
$$Q{'_i} = 1 - {\rho _{\max }} + {\rho _{\max }}\sum\limits_{j = 0}^{i - 1} {{a_{ij}}Q{'_j}}  \leq  1 - {\rho _{\max }} + {\rho _{\max }}C \leq  C$$
therefore we have obtained
$${Q_i} \leqslant \tilde Q{'_i} = \frac{{Q{'_i}\Theta }}{{1 - {\rho _{\max }}}} \leq  \frac{{C\Theta }}{{1 - {\rho _{\max }}}}$$
Finally, we prove that $\frac{\Theta }{{1 - {\rho _{\max }}}}$ can be upper bounded by a constant independent of $\Delta t$. In fact, thanks to the last three assumptions we have
$$C_0' \leqslant \frac{\Theta }{{1 - {\rho _{\max }}}} \leqslant \frac{{{C_2}{{(\Delta t)}^\alpha }}}{{{C_1}{{(\Delta t)}^\alpha }}} = \frac{{{C_2}}}{{{C_1}}}$$
in this case, $C$ is selected as $Q_0C_0'$ and we have
\begin{equation}
	Q_i \leq \frac{Q_0C_0'C_2}{C_1}:=C_0
\end{equation}
where the constant $C_0$ is independent of $i$ and $\Delta t$. 

As a consequence, 
\[
	\left\| {\frac{{DJ}}{{D\bt}}} \right\| = \left\| {\frac{{\partial J}}{{\partial {{\bu}^n}}}\frac{{D{{\bu}^n}}}{{D\bt}}} \right\| \leqslant \left\| {\frac{{\partial J}}{{\partial {{\bu}^n}}}} \right\|\left\| {\frac{{D{{\bu}^n}}}{{D\bt}}} \right\| = {C_0}\left\| {\frac{{\partial J}}{{\partial {{\bu}^n}}}} \right\|
\]
\end{proof}

To make the analysis simple, in the following corollaries, we assume the transport coefficients $b_1=b_2=0$. As we can see, the stability depends on how we discretize the system.

\begin{corollary}[Stability for the fractional advection diffusion equation]\label{coro:1}
	 The scheme proposed in \Cref{equ:tscheme} is stable with respect to $a$.
\end{corollary}
\begin{proof}
	In this case, we have
\begin{equation}
  \begin{cases}
	a_{n,k} = G_{n-k-1}-G_{n-k} & n-1\geq k\geq 1\\    
	a_{n,0} = -G_n & k = 0
\end{cases}
\end{equation}
Therefore
\begin{equation}
	\sum_{j=0}^{i-1}a_{ij} = 1+(i-1)^{1-\alpha}-(i+1)^{1-\alpha} < 1
\end{equation}
the first assumption is satisfied.

Note $A_i = 1+\frac{\Delta \tau^\alpha}{\Gamma(2-\alpha)}A$, where $A$ is the discretized advection diffusion operator and therefore
\begin{equation}
	\rho(A_i^{-1}) = \rho\left( \left(I+\frac{\Delta \tau^\alpha}{\Gamma(2-\alpha)}A\right)^{-1}\right) \leq  1- C(\Delta t)^\alpha
\end{equation}
for some $C>0$. Hence the second assumption is satisfied. 

For the third assumption, consider the diffusion coefficient $a$, given $\bv\in \RR^m$, and $\tilde\bv=(I+\Gamma(2-\alpha)\Delta \tau^\alpha A)^{-1}\bv$, we have
\begin{equation}
	\tilde\bv^T\frac{\partial \mathbf{m}^{n+1}}{\partial a} =\Gamma(2-\alpha) \Delta \tau^\alpha \sum_{i, j} \tilde v_{ij}\frac{{m_{i - 1,j}^{n+1} + m_{i + 1,j}^{n+1} + m_{i,j - 1}^{n+1} + m_{i,j + 1}^{n+1} - 4m_{i,j}^{n+1}}}{{{h^2}}}
\end{equation}
and therefore
\begin{equation}
	\left\|\tilde\bv^T\frac{\partial \mathbf{m}^{n+1}}{\partial a}\right\|\leq \Delta \tau^\alpha\frac{8\Gamma(2-\alpha) }{h^2} \|\tilde \bv\| \|\mathbf{m}^{n+1}\|
\end{equation}
The upper bound of $\|\mathbf{m}^{n+1}\|$ independent of $\Delta t$ is from standard estimation for the forward problem, see, for example, \cite{jiang2015new}. Therefore, we have
\begin{equation}
	\left\|\frac{\partial \mathbf{m}^{n+1}}{\partial a}\right\| \leq C\Delta t^\alpha
\end{equation}
for a constant $C$ independent of $\Delta t$. Hence the third condition is satisfied. Finally, note that 
\begin{equation}
	\left\|\frac{\partial \mathbf{m}^{n+1}}{\partial a}\right\|  = \Delta t^\alpha m_a^{n+1}
\end{equation}
for a $m_a^{n+1}$ independent of $\Delta t$, thus
\begin{equation}
	\max_{n}\left\|\frac{\partial \mathbf{m}^{n+1}}{\partial a}\right\|  =\Delta t^\alpha \max_{n} m_a^{n+1}
\end{equation}
In addition, $A$ is the discrete Laplacian operator, and therefore $A$ is symmetric. Assume the smallest eigenvalue of $A$ is $\lambda_{\min}>0$, then we have
$$\rho (A_i^{ - 1}) = \rho \left( {{{\left( {I + \frac{{\Delta {t ^\alpha }}}{{\Gamma (2 - \alpha )}}A} \right)}^{ - 1}}} \right) = \frac{1}{{1 + \frac{{\Delta {t ^\alpha }}}{{\Gamma (2 - \alpha )}}{\lambda _{\min }}}}$$
from which we have
\begin{equation}
	\dfrac{\mathop {\max }\limits_{0 \leqslant i \leqslant n} \left\| {\frac{{\partial {{\bu}^i}}}{{\partial \bt}}} \right\|}{1-\mathop {\max }\limits_{0 \leqslant i \leqslant n} \rho (A_i^{ - 1})} = \frac{{\Delta {t ^\alpha }\mathop {\max }\limits_n m_a^{n + 1}}}{{1 - \frac{1}{{1 + \frac{{\Delta {t ^\alpha }}}{{\Gamma (2 - \alpha )}}{\lambda _{\min }}}}}} \geqslant \Gamma (2 - \alpha )\frac{{\mathop {\max }\limits_n m_a^{n + 1}}}{{{\lambda _{\min }}}}
\end{equation}
therefore the last assumption is also satisfied. 

As a result, it is stable for the diffusion coefficient $a$.
\end{proof}

\newpage

\bibliographystyle{siamplain}
\bibliography{main}

\begin{thebibliography}{10}

\bibitem{whitepap56:online}
{\em White paper: Tensorflow control flow implementation}.
\newblock
  \url{http://download.tensorflow.org/paper/white_paper_tf_control_flow_implementation_2017_11_1.pdf}.
\newblock (Accessed on 03/12/2019).

\bibitem{antil2015fractional}
{\sc H.~Antil, E.~Otarola, and A.~J. Salgado}, {\em A fractional space-time
  optimal control problem: analysis and discretization$\backslash$}, arXiv
  preprint arXiv:1504.00063,  (2015).

\bibitem{baydin2018automatic}
{\sc A.~G. Baydin, B.~A. Pearlmutter, A.~A. Radul, and J.~M. Siskind}, {\em
  Automatic differentiation in machine learning: a survey}, Journal of Marchine
  Learning Research, 18 (2018), pp.~1--43.

\bibitem{biondi2012tomographic}
{\sc B.~Biondi and A.~Almomin}, {\em Tomographic full waveform inversion (tfwi)
  by combining full waveform inversion with wave-equation migration velocity
  anaylisis}, in SEG Technical Program Expanded Abstracts 2012, Society of
  Exploration Geophysicists, 2012, pp.~1--5.

\bibitem{bonnet2005inverse}
{\sc M.~Bonnet and A.~Constantinescu}, {\em Inverse problems in elasticity},
  Inverse problems, 21 (2005), p.~R1.

\bibitem{charpentier2001checkpointing}
{\sc I.~Charpentier}, {\em Checkpointing schemes for adjoint codes: Application
  to the meteorological model meso-nh}, SIAM Journal on Scientific Computing,
  22 (2001), pp.~2135--2151.

\bibitem{chen2000formulations}
{\sc Z.~Chen}, {\em Formulations and numerical methods of the black oil model
  in porous media}, SIAM Journal on Numerical Analysis, 38 (2000),
  pp.~489--514.

\bibitem{courty2003reverse}
{\sc F.~Courty, A.~Dervieux, B.~Koobus, and L.~Hasco{\"e}t}, {\em Reverse
  automatic differentiation for optimum design: from adjoint state assembly to
  gradient computation}, Optimization Methods and Software, 18 (2003),
  pp.~615--627.

\bibitem{farrell2013automated}
{\sc P.~E. Farrell, D.~A. Ham, S.~W. Funke, and M.~E. Rognes}, {\em Automated
  derivation of the adjoint of high-level transient finite element programs},
  SIAM Journal on Scientific Computing, 35 (2013), pp.~C369--C393.

\bibitem{giles2005using}
{\sc M.~B. Giles, D.~Ghate, and M.~C. Duta}, {\em Using automatic
  differentiation for adjoint cfd code development}, tech. report, Unspecified,
  2005.

\bibitem{girija2016tensorflow}
{\sc S.~S. Girija}, {\em Tensorflow: Large-scale machine learning on
  heterogeneous distributed systems},  (2016).

\bibitem{gockenbach2002automatic}
{\sc M.~S. Gockenbach, D.~R. Reynolds, and W.~W. Symes}, {\em Automatic
  differentiation and the adjoint state method}, in Automatic differentiation
  of algorithms, Springer, 2002, pp.~161--166.

\bibitem{gremse2016gpu}
{\sc F.~Gremse, A.~H{\"o}fter, L.~Razik, F.~Kiessling, and U.~Naumann}, {\em
  Gpu-accelerated adjoint algorithmic differentiation}, Computer physics
  communications, 200 (2016), pp.~300--311.

\bibitem{grote2010efficient}
{\sc M.~J. Grote and I.~Sim}, {\em Efficient pml for the wave equation}, arXiv
  preprint arXiv:1001.0319,  (2010).

\bibitem{gulian2018machine}
{\sc M.~Gulian, M.~Raissi, P.~Perdikaris, and G.~Karniadakis}, {\em Machine
  learning of space-fractional differential equations}, arXiv preprint
  arXiv:1808.00931,  (2018).

\bibitem{gcep2008}
{\sc J.~Harris, Y.~Quan, et~al.}, {\em Geologic storage of co2 subsurface
  monitoring of co2 sequestration in coal (gcep annual report 2008)}, tech.
  report, Stanford University, 2008.

\bibitem{gcep_2009}
{\sc J.~Harris, E.~Um, and T.~Zhu}, {\em Geologic storage of co2 subsurface
  monitoring of co2 storage in coal (gcep annual report 2009)}, tech. report,
  Stanford University, 2009.

\bibitem{hogea2008brain}
{\sc C.~Hogea, C.~Davatzikos, and G.~Biros}, {\em Brain--tumor interaction
  biophysical models for medical image registration}, SIAM Journal on
  Scientific Computing, 30 (2008), pp.~3050--3072.

\bibitem{jiang2017fast}
{\sc S.~Jiang, J.~Zhang, Q.~Zhang, and Z.~Zhang}, {\em Fast evaluation of the
  caputo fractional derivative and its applications to fractional diffusion
  equations}, Communications in Computational Physics, 21 (2017), pp.~650--678.

\bibitem{jiang2015new}
{\sc Y.~Jiang}, {\em A new analysis of stability and convergence for finite
  difference schemes solving the time fractional fokker--planck equation},
  Applied Mathematical Modelling, 39 (2015), pp.~1163--1171.

\bibitem{knopoff2013adjoint}
{\sc D.~A. Knopoff, D.~R. Fern{\'a}ndez, G.~A. Torres, and C.~V. Turner}, {\em
  Adjoint method for a tumor growth pde-constrained optimization problem},
  Computers \& Mathematics with Applications, 66 (2013), pp.~1104--1119.

\bibitem{leung2014overview}
{\sc D.~Y. Leung, G.~Caramanna, and M.~M. Maroto-Valer}, {\em An overview of
  current status of carbon dioxide capture and storage technologies}, Renewable
  and Sustainable Energy Reviews, 39 (2014), pp.~426--443.

\bibitem{li2010convolutional}
{\sc Y.~Li and O.~Bou~Matar}, {\em Convolutional perfectly matched layer for
  elastic second-order wave equation}, The Journal of the Acoustical Society of
  America, 127 (2010), pp.~1318--1327.

\bibitem{martin2009unsplit}
{\sc R.~Martin and D.~Komatitsch}, {\em An unsplit convolutional perfectly
  matched layer technique improved at grazing incidence for the viscoelastic
  wave equation}, Geophysical Journal International, 179 (2009), pp.~333--344.

\bibitem{martins2002coupled}
{\sc J.~R. Martins}, {\em A coupled-adjoint method for high-fidelity
  aero-structural optimization}, tech. report, STANFORD UNIV CA DEPT OF
  AERONAUTICS AND ASTRONAUTICS, 2002.

\bibitem{maryshev2013adjoint}
{\sc B.~Maryshev, A.~Cartalade, C.~Latrille, M.~Joelson, and M.-C. N{\'e}el},
  {\em Adjoint state method for fractional diffusion: parameter
  identification}, Computers \& Mathematics with Applications, 66 (2013),
  pp.~630--638.

\bibitem{mavko2009rock}
{\sc G.~Mavko, T.~Mukerji, and J.~Dvorkin}, {\em The rock physics handbook:
  Tools for seismic analysis of porous media}, Cambridge university press,
  2009.

\bibitem{oberai2003solution}
{\sc A.~A. Oberai, N.~H. Gokhale, and G.~R. Feij{\'o}o}, {\em Solution of
  inverse problems in elasticity imaging using the adjoint method}, Inverse
  problems, 19 (2003), p.~297.

\bibitem{pratt1999seismic}
{\sc R.~G. Pratt}, {\em Seismic waveform inversion in the frequency domain,
  part 1: Theory and verification in a physical scale model}, Geophysics, 64
  (1999), pp.~888--901.

\bibitem{tarantola1984inversion}
{\sc A.~Tarantola}, {\em Inversion of seismic reflection data in the acoustic
  approximation}, Geophysics, 49 (1984), pp.~1259--1266.

\bibitem{thomas2013numerical}
{\sc J.~W. Thomas}, {\em Numerical partial differential equations: finite
  difference methods}, vol.~22, Springer Science \& Business Media, 2013.

\bibitem{tromp2005seismic}
{\sc J.~Tromp, C.~Tape, and Q.~Liu}, {\em Seismic tomography, adjoint methods,
  time reversal and banana-doughnut kernels}, Geophysical Journal
  International, 160 (2005), pp.~195--216.

\bibitem{vazquez2017mathematical}
{\sc J.~L. V{\'a}zquez}, {\em The mathematical theories of diffusion: Nonlinear
  and fractional diffusion}, in Nonlocal and nonlinear diffusions and
  interactions: new methods and directions, Springer, 2017, pp.~205--278.

\bibitem{virieux2009overview}
{\sc J.~Virieux and S.~Operto}, {\em An overview of full-waveform inversion in
  exploration geophysics}, Geophysics, 74 (2009), pp.~WCC1--WCC26.

\end{thebibliography}
\end{document}